\documentclass[11pt]{article}

\usepackage{amsmath,amssymb,amscd,amsfonts, amsthm}
\usepackage{amsmath}
\usepackage{amssymb}
\usepackage{graphicx}
\usepackage{caption}
\usepackage{subcaption}
\usepackage{graphicx,latexsym}

\usepackage{anysize}

\usepackage{color}
\usepackage{xcolor}

\usepackage{epsfig}

 \usepackage{bm}

\usepackage[bookmarksopen, colorlinks,
  linkcolor =blue, urlcolor =red, citecolor =red, menucolor =blue]{hyperref}

\usepackage{chngcntr}

\usepackage{placeins}
\usepackage{enumitem}

\newcommand{\numparam}[1]{\textcolor{black}{#1}}
\newcommand\EH[1]{\textcolor{cyan}{EH: #1}}

\newcommand\R{\mathbb{R}}

\newcommand\norm[2][]{{\left\lVert#2\right\rVert_{#1}}}

\newtheorem{theo}{Theorem}
\newtheorem{lemma}[theo]{Lemma}

\newtheorem{prop}[theo]{Proposition}
\newtheorem{defin}{Definition}
\newtheorem{remark}{Remark}

\def\eps{{\varepsilon}}

\def\M{{\cal M}}
\def\W{{\cal W}}
\def\R{{\cal R}}
\def\K{{\cal K}}

\def\F{{\cal F}}
\def\Fa{{\cal F}_\alpha}

\def\La{{\cal L}_\alpha}
\def\Lak{{\cal L}_{\alpha_j}}

\def\.{{\;}}
\def\bv{{\tt{BV}}}

\def\F{{\cal F}}
\def\R{{\cal R}}
\def\W{{\cal W}}
\def\K{{\cal K}}

\def\Fa{{\cal F}_\alpha}
\def\La{{\cal L}_\alpha}

\begin{document}
\setcounter{footnote}{1}

\title{A piecewise constant levelset approach for semi-blind deconvolution:
       Application to barcode decoding}

\author{
A.~De\,Cezaro%
\thanks{Institute\,of\,Math.\,Statistics\,and\,Physics,
Federal\,Univ.\,of\,Rio\,Grande,\,Av.\,Italia\,km\,8,\,96201-900 Rio\,Grande,\,Brazil}
\ \ \
E.~Hafemann%
\thanks{Department of Mathematics, Federal University of St.\,Catarina,
       P.O. Box 476, 88040-900 Floripa, Brazil}
\thanks{Fachbereich Mathematik, Universit\"at Hamburg, Bundesstraße 55, 20146 Hamburg, Germany}
\ \ \
A.~Leit\~ao$^\ddag$}

\date{\small\today}

\maketitle

\begin{abstract}
We consider a semi-blind deconvolution problem arising in the decoding of blurred linear
barcodes.
Building on the Piecewise Constant Level Set (PCLS) framework introduced in [De\,Cezaro et
al., Inv.\,Probl., 29 (2013), 015003], we propose and analyze a solution method based on
augmented Lagrangians to obtain stable approximate solutions to the corresponding inverse
problem with respect to noisy measurements.
We establish the existence of generalized multipliers for the augmented Lagrangian functional
under consideration, as well as the absence of duality gaps. These results provide the
theoretical foundation required to prove regularization properties of the approximate
solutions produced by the proposed strategy.
Furthermore, we present an associated ADMM-type iterative scheme for the explicit computation
of approximate barcodes. Numerical experiments are carried out for various variance values
(responsible for the blurred effect) and several levels of noise, validating the effectiveness
of the proposed method.
\end{abstract}

\noindent {\small {\bf Keywords.}
Semi-blind deconvolution, Barcode decoding, Piecewise constant levelset, Augmented Lagrangian
.}

\medskip
\noindent {\small {\bf AMS Classification:} 65J20, 47J06, 65R32.}

\renewcommand{\thefootnote}{} \footnotetext{Emails:
\href{mailto:decezaromtm@gmail.com}{\tt decezaromtm@gmail.com}, \
\href{mailto:eduardo.hafemann@uni-hamburg.de}{\tt eduardo.hafemann@uni-hamburg.de}, \
\href{mailto:acgleitao@gmail.com}{\tt acgleitao@gmail.com}.}
\renewcommand{\thefootnote}{\arabic{footnote}}

\section{Introduction} \label{sec:1}

A {\em barcode} (or bar code) image serves as a visual representation of data.
A linear (or 1D) barcode, for example, comprises parallel lines and bars with
differing widths and gaps, encoding specific information.
Different barcode symbologies do exist, each with its unique pattern for
encoding data, e.g., linear type codes (e.g., ISBN, MSI, Royal Mail-4,
2-of-5 IATA, UPC and Code-128) and 2D type codes (including QR-Code,
Data Matrix, PDF417, MaxiCode and Aztec); visit the TEC-IT website
\href{https://www.tec-it.com/en/support/knowbase/Default.aspx}{https://www.tec-it.com/}
for a detailed presentation on barcode symbologies.

\vspace{-0.3cm}\paragraph{Barcode decoding:}
Barcodes find extensive application across diverse industries, serving multiple purposes, e.g.,
inventory management, retail sales, tracking shipments and ticketing. When scanned,
a linear barcode pattern of bars and spaces is translated into alphanumeric characters,
allowing the encoded information to be utilized.
The swift and precise scanning by barcode readers plays a crucial role in the
application of this technology.

\noindent
\begin{minipage}[t]{.4\textwidth}
{\ \ \ \ \,The decoding of the information in a linear barcode is typically
due to an optical scanner equipped with a light detector. A narrow laser beam
is emitted, which illuminates the barcode --Figure~1 (right). 
The amount of light detected, namely a voltage reading, generates a
1D electrical signal with peaks corresponding to the white parts (due
high light reflection) and valleys corresponding to the black bars (due
to little light reflection); see Figure~1 (left).}
\end{minipage}
\hfil\noindent
\begin{minipage}[t]{.7\textwidth}
\vskip-0.3cm \hskip-0.7cm
\centerline{\includegraphics[width=0.6\textwidth]{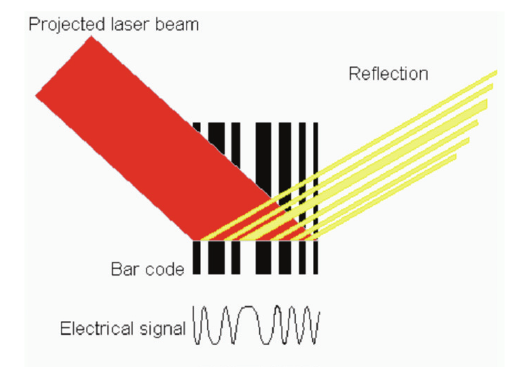}
\hskip-2.0cm\includegraphics[width=0.45\textwidth]{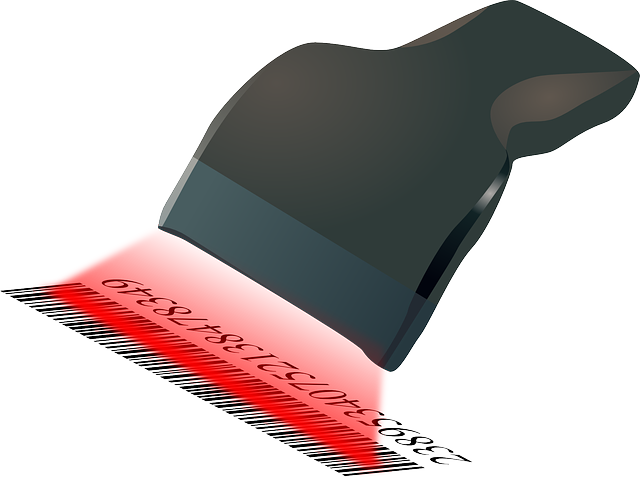}}
\vskip0.2cm \hskip-0.8cm
\centerline{\small Figure 1: Barcode optical scanner (source:\,pixabay.com).}
\end{minipage}
\medskip

\noindent\textbf{Mathematical model:} \
In this 1D setup the barcode can be described as a piecewise constant function
$u$ taking only two different known values (namely, zero for black and one for
white) in a given connected bounded interval $\Omega \subset \mathbb{R}$. More
precisely, we assume the existence of an open mensurable set
$D_1 \subset\subset \Omega$ such that
\begin{equation} \label{eq:u-pwc}
u(x) = 0 , \ x \in D_1 \quad {\rm and}
\quad u(x) = 1, \ x \in D_2 := \Omega \setminus D_1 \, .
\end{equation}

The laser beam has a variable intensity profile of a finite width. Consequently,
the voltage readings of the barcode by the scanner (available measurements) suffer
non-local effects.
Since the black bars in the barcode $u$ are finite ($\Omega$ is bounded), the
non-local effects far from the barcode location do not influence the scanning
process. Hence, in the mathematical model considered in this article, we assume
that the effects of the laser beam on the barcode are modeled by a point-spread
function (PSF) with compact support
\begin{align} \label{def:PSF}
K(x) \ = \ K(\gamma, \sigma)(x) \ = \ \chi_{\Omega}(x) \,
\frac{\gamma}{\sigma \sqrt{2\pi}} \, e^{-\frac{x^2}{2\sigma^2}} \, ,
\end{align}
where the variance $\sigma^2$ models the distance between the scanner and the
bar code surface (the longer the distance from the scanner to the barcode, the
larger the variance) and $\gamma > 0$ is the amplitude of the Gaussian kernel $K$
($\gamma$ encapsulates the process of converting light energy that interacts
with the barcode into the measurement).
Additionally, $\chi_\Omega$ denotes the characteristic function of $\Omega$.
Since the coefficients $(\sigma,\gamma)$ are not known, the kernel $K$ in
\eqref{def:PSF} is only partially known.

We consider the setting in which the amplitude $\gamma$ of the kernel $K$ is known.
The inverse problem under investigation consists of identifying both the variance
$\sigma^2$ and the barcode $u$.

The interaction of the barcode with the laser beam is modeled by the nonlinear operator
\begin{align} \label{def:B}
B: & L^2(\Omega) \times L^1(\Omega) \longrightarrow L^2(\Omega)\nonumber\\
   & \qquad (K,u) \qquad \longmapsto B(K,u) := K \ast u,
\end{align}
where the symbol $\ast$ represents the convolution operation; the functions
$u \in L^1(\Omega)$ and $K \in L^2(\Omega)$ are given as above. The signal
(or sensor reading)
$y \in L^2(\Omega)$ resulting from this interaction is described by the
well established mathematical model $y = K \ast u$ (see, e.g., \cite{E04,ES2012}).

\vspace{-0.3cm}\paragraph{Underlying inverse problem:}
The transformation of the barcode into the signal that is actually detected
by the scanner depends on a variety of factors such as: distance between the
optical scanner and the surface where the barcode appears (the farther away
the scanner is located, the more distorted/blurred the detected signal), the
illumination of the optical sensor (inside the barcode scanner), among others.
Therefore, in practice, only a noise version $y^\delta \in L^2(\Omega)$ of the
signal $y$ is available. We assume that the level of noise $\delta$ in the
signal can be estimated by
\begin{equation} \label{eq:noise}
\big\| y - y^\delta \big\|_{L^2(\Omega)} \ \leq \ \delta \, .
\end{equation}

Summarizing, the inverse problem we are dealing with is the one of identifying the
pair $(K,u)$ in the operator equation
\begin{equation} \label{eq:ip}
B(K,u) \ = \ y^\delta \, ,
\end{equation}
from noise measurements $y^\delta$ of the signal $y$ satisfying \eqref{eq:noise}.
An efficient method designed for solving \eqref{eq:ip}, \eqref{eq:noise} should take
into account the following {\em a priori} information: (i) available partial knowledge
of the PSF $K$, which is known to be of the form \eqref{def:PSF}; (ii) the fact that
$u$ (the barcode) is a piecewise constant function satisfying \eqref{eq:u-pwc}.

The barcode identification problem considered here differs from the standard image
analysis problems of image deblurring \cite{BBR18,ES2012} and blind deconvolution
\cite{Gennip15, LV2021, RSCHM2021}.
Here the convolution kernel $K$ is not completely known and the scalar parameter
$\sigma$ must be identified together with the barcode $u$. In other words, this
linear barcode identification problem can be viewed as a
{\em semi-blind deconvolution problem} \cite{E04, BDR18, LEZX2014}.

\vspace{-0.3cm}\paragraph{Literature overview:}
There is a vast amount of literature on barcode decoding. In this paragraph,
we focus on some inverse problems (and imaging) approaches that are relevant to the
content covered in this manuscript.

Commercial deblurring techniques for hand scanners are often based on edge
detection, which plays a relevant role in barcode decoding.
These techniques relate to the search for local extrema of the derivative of the
noisy signal $y^\delta$ (zero crossings of the second derivative provide edge
features for the classification of blurred images); however, it is known to be
highly sensitive to small changes in the data (ill-posedness) \cite{EngHanNeu96}.
Another drawback of this strategy is related to the standard deviation $\sigma$.
If $\sigma$ is "large" or if the support size of $K$ is very large when compared
to the smallest length scale of the minimum thickness of bars in the barcode,
then the extremes of the derivative of $y^\delta$ might not correspond to the
edges of the barcode to be decoded, even when $\delta=0$.
In \cite{KJ2005} the edge localization error caused by speckle noise (e.g.,
when a barcode signal is corrupted by additive noise which is a weakly stationary
random process) is estimated.
In \cite{JP1994} the severely blurred image scenario is considered;
statistical pattern recognition is used to classify the peaks as (UPC)
barcode characters.

In \cite{E04} the semi-blind decoding of linear barcodes in \eqref{eq:ip}, with
$K$ as in \eqref{def:PSF}, is considered. Energy minimization strategies with
total-variation (TV) penalization are proposed and analyzed for the stable
barcode decoding problem~\eqref{eq:ip} with noise measurements $y^\delta$
(e.g., existence and uniqueness of minimizers $(\gamma, \sigma, u)$ for a
Tikhonov functional based on $TV$-penalization is proven in \cite[Proposition~4]{E04}).
The recovered barcode is numerically obtained as the $\Gamma$ limit of a
"diffuse interface" approximation of an energy minimization strategy, with a
distinct level of blurring $\sigma$.

Other approaches to semi-blind barcode decoding are considered in
\cite{LEZX2014, RSCHM2021}. In \cite{LEZX2014}, for 1D barcodes, the authors
investigate the effect of nonstationary out-of-focus blurring, i.e. when
$(\gamma,\sigma) = (\gamma(s),\sigma(s))$ ($s$ is used to parameterize
the depth map of the barcode by its arc length \cite[Figure~1]{LEZX2014}).
An algorithm based on the minimization of an energy functional (using
penalization that includes the $l_0$ gradient norm of $u$) is proposed.
A low-dimensional representation of the PSF is estimated using the
Levenberg–Marquardt method. Once the PSF is obtained, a deblurred image
is computed by solving a quadratic program (a $[0,1]$ box constraint is
used to enforce a binary signal).
In \cite{RSCHM2021} a penalization based on Kullback-Leibler divergence
is considered for semi-blind decoding of UPC-A and QR barcodes.
The proposed approach, which incorporates and exploits barcode
symbologies, uses a primal-dual strategy.
The method does not aim to directly locate the cleaned image; instead, it
determines the probability density function of the image across all binary
arrays.
Subsequently, the cleaned image is derived from its threshold expectation.

In \cite{ES2012} a deblurring problem related to \eqref{eq:ip} is considered.
The blur kernel $K$ is assumed to be known. The authors provide conditions,
depending on $\sigma$ and the additive noise level, under which the variational
method proposed in \cite{E04} is able to recover the barcode. 
The analysis in \cite{ES2012} is extended in \cite{SG2022} for camera-reading
decoding of UPC barcodes, where the authors prove that, if the UPC barcode has the
narrowest bar corresponding to two thirds the size of the camera pixel, the
encoded message in the barcode can be uniquely determined.

In \cite{ISW2013} a symbology-based reconstruction of the UPC barcode decoding problem
is investigated.
The symbology of the barcode is incorporated directly into the reconstruction algorithm
and a sparse representation of the barcode is reconstructed with respect to this dictionary
(this approach reduces the number of degrees of freedom in the problem).

In \cite{CG10} total variation-based energy penalization is used to recover a blurred
1D barcode. Tikhonov functionals defined over all possible barcodes with fidelity to a
convoluted signal of a barcode and regularized by a $BV$-term. The fidelity terms consist
of the $L^2$-distance to the measured signal (or preceded by deconvolution).
The authors established sufficient conditions for the support of $K$, the level of
blurring $\sigma$, the size of the bars and the regularization parameters where the
underlying bar code is the unique minimizer. In \cite{CGO11} the same authors apply
total variation penalization and $L^1$-fidelity to the recovery of blurred 2D barcodes.

\vspace{-0.3cm}\paragraph{Main contributions:}
We introduce and study a regularization technique that relies on the primal solutions of a
special augmented Lagrangian functional linked to the underlying inverse problem. Specifically,
we establish the existence of generalized multipliers for the augmented Lagrangian functional
under study and show that no duality gap occurs. These findings form the theoretical basis
needed to prove the regularization behavior of the approximate solutions generated by the
proposed approach.

We also provide an ADMM-type iterative algorithm for the explicit computation of approximate
barcodes. Numerical experiments are conducted for a range of variance parameters (governing
the blurring effects) and noise intensities, confirming the effectiveness of the proposed
method.

\section{PCLS ansatzt, slack variable, augmented Lagrangian} \label{sec:level-set-formulation}

In this section we discuss three important tools for computing stable approximate
solutions $(K_\alpha^\delta,\phi_\alpha^\delta)$ to the inverse problem \eqref{eq:noise},
\eqref{eq:ip}, namely: 1) A PCLS ansatz \eqref{eq:def-pcls} used to represent the
unknown barcode; 2) the introduction of a slack variable regarading $K$ (this is a
quintessential step that enables the use of augmented Lagrangians); 3) An augmented
Lagrangian functional \eqref{eq:tikhonov-pcls3} that allows the efective computation
of the desired approximate solutions.

\paragraph{Piecewise Constant Level-Sets:}
In \cite{CLT13} a PCLS ansatzt is used to represent the solution of a nonlinear ill-posed
operator equation, under the {\em a priori} assumption that the sought solution is a
piecewise constant function taking only two possible (unknown) values
\cite{TaiChan04, NTAE07, CL11, CLT13, Agnelli2018}.
According to this approach, a function $u \in \bv(\Omega)$ satisfying $u(x) \in \{c^1, c^2\}$
a.e. in $\Omega$ is represented by
\begin{equation} \label{eq:def-pcls}
u \ = \ c^1 \psi_1(\hat{\phi})  +  c^2 \psi_2(\hat{\phi}) \ =: \ P_{pc}(\hat{\phi}) \, ,
\end{equation}
where $\hat{\phi}: \Omega \to \mathbb R$ is the piecewise constant levelset function defined
by $\hat{\phi}(x) := i-1$, $x \in D_i$, for $i = 1,2$, and $\psi_1(t) = 1 - t$, $\psi_2(t) = t$
are real functions.

\begin{remark} \label{remark:PCLS-bar-codes}
In the setting discussed in Section~\ref{sec:1} the function $u$ represents the unknown
barcode, which is known to satisfy $u(x) \in \{ 0,1 \}$ a.e. in $\Omega$ (see \cite{E04,
Gennip15, CGO11}).
This is equivalent to choosing $c^1 = 0$ and $c^2 = 1$ in \eqref{eq:def-pcls}. With this
particular choice, the operator $P_{pc}$ in \eqref{eq:def-pcls} reduces to the identity
operator.
\end{remark}

Here the PCLS ansatzt \eqref{eq:def-pcls} is used to represent the component
$u$ of the solution pair $(K,u)$ of the inverse problem \eqref{eq:ip}, \eqref{eq:noise}.
This allow us to incorporate to the operator equation \eqref{eq:ip} the {\em a priori}
information described in \eqref{eq:u-pwc}.
In view of Remark~\ref{remark:PCLS-bar-codes}, the inverse problem of barcode decoding
becomes the one of identifying the pair $(K,\phi)$ in the system
\begin{equation} \label{eq:inv-probl-fpc}
\left\{ \begin{array}{l}
B(K, \phi ) \ = \ y^\delta \\[1ex]
\W(\phi) \ = \ 0
\end{array} \right. ,
\end{equation}
from noise measurements $y^\delta$ satisfying \eqref{eq:noise}. Here
$K$ has the same meaning as in \eqref{eq:ip} and $\phi \in L^{4n}(\Omega)\cap
\bv(\Omega)$ is a Piecewise Constant LevelSet function. Additionally, $\W:
L^{4n}(\Omega) \cap \bv(\Omega) \to L^2(\Omega)$ is the nonlinear double-well
operator defined by $\W(\phi): = \phi^n(\phi-1)^n$, for $n\in \{1, 2, \cdots\}$
(see, e.g., \cite{TaiChan04, NTAE07} for details).

\begin{remark} \label{remark:1}
Notice that the level set function $\hat{\phi} \in L^{4n}(\Omega) \cap \bv(\Omega)$ in
\eqref{eq:def-pcls} satisfies $\hat{\phi}(x) := i-1$, $x \in D_i$, $i = 1, 2$;
therefore $\W(\hat{\phi}) = 0$.
On the other hand, if $\W(\phi) = 0$ for some function $\phi \in L^{4n}(\Omega)
\cap \bv(\Omega)$
then $\phi(x) \in \{0,1\}$, a.e. in $\Omega$ (see Lemma~\ref{lemma:calK-Ppc}~$(ii)$
below).
Consequently, any solution $(K,\phi)$ of system \eqref{eq:inv-probl-fpc} is also a
solution of \eqref{eq:ip} having piecewise constant component $\phi$.
\end{remark}

In order to obtain stable approximate solutions to problem \eqref{eq:inv-probl-fpc},
\eqref{eq:noise} we consider the minimizers $(K_\alpha^\delta,\phi_\alpha^\delta)$,
if exist, of the (regularized) constrained optimization problem
\begin{equation} \label{eq:inv-probl-copt1-tilde}
\left\{ \begin{array}{l}
   \min F_{\alpha}(K,\phi):= \| B(K, \phi ) - y^\delta \|^2_{L^2(\Omega)} 
                    \, + \, \alpha R(K, \phi) \\[1ex]
   \mbox{s.t. \ } \W(\phi) = 0
\end{array} \right. ,
\end{equation}
where $F_{\alpha}: L^2(\Omega) \times (L^{4n}(\Omega) \cap \bv(\Omega)) \to \mathbb{R}_+$
and the penalization functional 
$R:L^2(\Omega) \times (L^{4n}(\Omega) \cap \bv(\Omega)) \to \mathbb{R}_+$ is defined by

\begin{equation} \label{eq:def-R-tilde}
R(K,\phi) \ := \beta_1\norm{K - K_0}^2_{H^1(\Omega)} + \beta_2\ \| \phi \|^2_{L^{4n}} + \beta_3 |\phi |_\bv \, .
\end{equation}
Here $K_0 \in H^1(\Omega)$ is some \textit{a priori} and $\beta_j > 0$ for $j=1,2,3$
are scaling factors. Furthermore, $\alpha > 0$ plays the role of a regularization
parameter as in classical Tikhonov regularization \cite{EngHanNeu96}.

\begin{remark} \label{rm:existence}
In view of the continuity of the bilinear operator $B(\cdot,\cdot)$, see
Proposition~\ref{th:B-continuos}, the existence of $(K_\alpha^\delta,\phi_\alpha^\delta) =$
argmin $F_{\alpha}(K,\phi)$, follows form standard compactness arguments \cite{EngHanNeu96}.
However, there is no guarantee that $\phi_\alpha^\delta$ satisfies the desired constraint
$\W(\phi_\alpha^\delta) = 0$.
Notice that this constraint cannot be enforced by standard projection strategies, since
the levelset $\{\phi \in L^{4n}(\Omega),\, : \, \W(\phi) = 0\}$ is not a convex subset of
$L^{4n}(\Omega)$.

A well established technique in the optimization literature \cite{Ber82, RockWets98} is to solve
the constraint optimization problem~\eqref{eq:inv-probl-copt1-tilde} is by means of
duality theory (also known as the Lagrangian approach). 
In the case where $K$ is known, i.e. the deblurring problem, a regularization approach based
on the augmented Lagrangian formulation for problem \eqref{eq:inv-probl-copt1-tilde} is possible
(following the lines of the analysis presented in Section~\ref{sec:convergence-analysis}).
In particular, the existence of a generalized Lagrangian multiplier $\lambda_K$ supporting
an exact penalty (in this case $\W(\phi)=0$) follows,  with the necessary adaptations,
similarly to the one presented below in Subsection~\ref{subsec:existence}.
However, the $K$-dependence of such a Lagrangian multiplier $\lambda_K$ makes the ongoing
theory intractable, as far as we know, if $K$ is also unknown.   
\end{remark}

\paragraph{Introduction of a slack variable:}
Notice that to each solution $(K_\alpha^\delta,\phi_\alpha^\delta)$ of
\eqref{eq:inv-probl-copt1-tilde}, there corresponds a solution of the optimization problem
\begin{equation} \label{eq:inv-probl-copt}
\left\{ \begin{array}{l}
   \min  \F_{\alpha}(\K,K,\phi):= \| B(K, \phi ) - y^\delta \|^2_{L^2(\Omega)} 
                    \, + \, \alpha \R(\K, K, \phi) \\[1ex]
   \mbox{s.t. \ } \W(\phi) = 0 \ \mbox{ and } \ M(\K,K) = 0                  
\end{array} \right.\,,
\end{equation} 
where the functionals $\F_{\alpha}$ and $\R$ map from $X := L^2(\Omega) \times
L^2(\Omega) \times (L^{4n}(\Omega) \cap \bv(\Omega))$ to $\mathbb{R}_+$, with
$\R(\K, K, \phi) := R(K,\phi) + \beta_1\norm{\K - K_0}_{H_1(\Omega)}$. Moreover,
$M:L^2(\Omega) \times L^2(\Omega) \to L^2(\Omega)$ is defined by $M(\K,K) :=
\K - K$ and $\W$ is defined as in \eqref{eq:inv-probl-fpc}. 

In problem \eqref{eq:inv-probl-copt} the slack variable $\K$ is introduced
(in \cite{BDR18} a similar technique is used for semi-blind regularization).
This strategy is known from the Alternation Direction Method of Multipliers
(ADMM) \cite{BDR18} and from the equality/inequality constraint optimization theory
\cite{Ber82, RockWets98}.
It has the disadvantage of increasing the number of unknowns (when compared
to \eqref{eq:inv-probl-copt1-tilde}). However, the formulation in
\eqref{eq:inv-probl-copt} has the advantage of incorporating equality constraints
in all problem unknowns $(\K,K,\phi)$. This characteristic is quintessential
in the derivation of a complete convergence analysis for the regularization method
based on the minimizers of \eqref{eq:inv-probl-copt} (see Remark~\ref{rm:existence}
and Section~\ref{sec:convergence-analysis} for details).

\paragraph{Introduction of augmented Lagrangians:}
The duality theory followed in this manuscript starts with the introduction of the augmented Lagrangian functional
\begin{multline} \label{eq:tikhonov-pcls3}
L_\alpha(\K, K,\phi;\lambda,\mu) \ := \
   \| B(K,\phi) - y^\delta \|^2_{L^2(\Omega)} 
   + \alpha \R(\K,K. \phi) \\
   + \langle \lambda \, , \big( M(\K,K), \W(\phi) \big) \rangle
   + \mu \| \big( M(\K, K), \W(\phi) \big) \|
\end{multline}
(the notation $\langle\cdot,\cdot\rangle$ and $\|\cdot\|$ denote the inner product
and norm in $L^2(\Omega) \times L^2(\Omega)$ respectively).
Here the primal variables are $(\K, K, \phi) \in X$; the dual variables
$(\lambda, \mu) \in (L^2(\Omega))^2 \times \mathbb{R}_{+}$ are such that: $\lambda$
can be interpreted as a ``generalized Lagrange multiplier" and $\mu$ is a penalty
factor that allows one to establish a duality relation for problems of non-convex
type. This particular augmented Lagrangian functional is also known as the sharp
Lagrangian \cite[Chapter~11, Section K$^*$]{RockWets98}.  

The Lagrangian functional in \eqref{eq:tikhonov-pcls3} converts the (primal)
constrained optimization problem~\eqref{eq:inv-probl-copt} into a family of
unconstrained sub-problems (depending on $(\lambda, \mu)\in (L^2(\Omega))^2
\times \mathbb{R}_{+}$), so that the information provided by the
constraints $\W(\phi) = 0$ and $M(\K,K)=0$ is incorporated in the Lagrangian
$L_\alpha$.
%
%
%
This approach is effective when the subproblems are simpler to solve than
the original constrained problem and the Lagrangian function is able to
provide the zero duality gap property (see Subsection~\ref{subsec:existence}).
Since $\W(\phi)$ is nonlinear, the zero duality gap for the constrained
optimization problem \eqref{eq:inv-probl-copt} does not follow straightforward
(see Subsection~\ref{subsec:existence}).

The idea here is to compute regularized solutions $(K_\alpha^\delta,\phi_\alpha^\delta)$,  as the minimizer (primal solutions $(\K_\alpha^\delta, K_\alpha^\delta,\phi_\alpha^\delta)$, with $\W(\phi_\alpha^\delta)=0$ and $M(\K_\alpha^\delta,K_\alpha^\delta)=0$,  of the augmented Lagrangian~\eqref{eq:tikhonov-pcls3} associated to the constrained optimization problem~\eqref{eq:inv-probl-copt}.

\section{Convergence analysis}\label{sec:convergence-analysis}

We denote by $Ad \subset X$ the set of admissible functions defined by
$Ad:=\{(\K,K,\phi) \in X\}$, subject to $\W(\phi)=0$, $M(\K, K) = 0$ and
$K \ast \phi \in L^2(\Omega)$.
In the following, we present the main assumptions used in these notes.

\noindent
{\bf (A1)} $\Omega \subseteq \mathbb{R}$, 
is bounded. 
\newline
{\bf (A2)} $\alpha$, $\beta_j$ denote positive parameters.
\newline
{\bf (A3)} There exists $\hat{\sigma} >0$,  $\hat{K} = K(\hat{\sigma}) \in L^2(\Omega)$ and  $\hat{u} \in \bv(\Omega) \cap L^\infty(\Omega)$ satisfying $B(\hat{K},\hat{u}) = y$. Moreover, there exists a function $\hat{\phi} \in \bv(\Omega) \cap L^{4n}(\Omega)$ such that
$\hat{\phi} = \hat{u}$ and $ \W(\hat{\phi}) = 0$. 
\smallskip

 

\subsection{Existence of an exact penalty representation}\label{subsec:existence}

The Tikhonov functional $\Fa$ is not convex, so the classical Lagrange multiplier methods \cite{RockWets98} cannot be used to solve \eqref{eq:inv-probl-copt}. This is because it is uncertain whether a zero duality gap property can be established for nonconvex optimization problems \cite{RockWets98}.


In this subsection, we follow similar ideas presented in \cite{Agnelli2018, CLT13} to prove the existence of an exact penalty representation (see Definition~\ref{def:exact-penalty}) based on the abstract convexity framework introduced in \cite{RockWets98}. It turns out that this property implies the zero-duality gap. Furthermore, it becomes the key ingredient to prove the well-posedness, the convergence, and stability of approximate solutions given by the augmented Lagrangian \eqref{eq:tikhonov-pcls3} (see Section~\ref{subsec:doble-well-convergence-analysis}).

We introduce some notation and functions related to the augmented Lagrangian approach that are necessary for the upcoming analysis.

\begin{defin} \label{def:vector-pc}
Let $\Fa$, $\W$ and $M$ be defined in the \eqref{eq:inv-probl-copt} and the Assumptions~\textbf{A1)} - \textbf{A3)} holds true.

\begin{enumerate}
\item $\Gamma: (L^2(\Omega))^2  \rightrightarrows L^2(\Omega)$ is the function with set values that satisfies $\Gamma(z) :=
\{ (\K,K, \phi) \in X \,:\,(M(\K,K), \W(\phi)) = z \}$, for each $z \in (L^2(\Omega))^2$. 

\item The indicator function of a set $A$ is defined by $\delta_A(\eta) := 0$, if $\eta \in A$ and $\delta_A(\eta) := +\infty$, otherwise.

\item We define $\widetilde{\Fa}(\K,K,\phi) := \Fa(\K,K,\phi)$, if $(\K,K,\phi) \in X \cap \Gamma(0)$ and $\widetilde{\Fa}(\K,K,\phi) := +\infty$, otherwise.

\item A dualizing parametrization function for
$\widetilde{\Fa}$ is chosen in the following way $f: X 
\times (L^2(\Omega))^2 \rightarrow \mathbb{R}\cup \{+ \infty\}$, with $f(\K,K,\phi, z) :=
\Fa(\K,K,\phi) + \delta_{\Gamma(z)}(\K,K,\phi)$ if $(\K, K,\phi) \in X$ and $f(\K,K,\phi, z) = +\infty$, otherwise. 
The function $f$ satisfies $f(\K,K,\phi,0) = \widetilde{\Fa}(\K,K,\phi)$, for each $(\K,K,\phi) \in X$.

\item The perturbation function (of the primal problem) related to this duality parametrization is given by $\theta: (L^2(\Omega))^2 \to
\mathbb{R}$, where $\theta(z) := \inf_{(\K,K,\phi) \in X} f(\K,K,\phi,
z)$. 

\item\label{item-coupling} The coupling function $\rho: (L^2(\Omega))^2\times (L^2(\Omega))^2 \times
\mathbb{R}_+ \to \mathbb{R}$ is defined by $\rho(z,\lambda,\mu) := -
\langle \lambda , z \rangle - \mu\,\norm{z}$.

\item The augmented Lagrangian functional induced by the coupling function $\rho$ reads as 
\begin{equation}\label{eq:L}
\La (\K,K,\phi;\lambda,\mu)  =  \textstyle \inf\limits_{z \in (L^2(\Omega))^2} \{ f(\K,K,\phi,z) - \rho(z,\lambda,\mu) \} \,.
\end{equation}

\item The dual function $Q: (L^2(\Omega))^2\times \mathbb{R}_+  \to \mathbb{R}$ is defined by
$$Q(\lambda,\mu) := \inf\limits_{(\K,K,\phi) \in X} \La (\K,K,\phi;\lambda,\mu)$$ 
and the dual  problem is stated as
\[ \text{maximize } Q(\lambda,\mu) \quad \text{subject to  } (\lambda,\mu) \in (L^2(\Omega))^2\times \mathbb{R}_+ . \]
\end{enumerate}
\end{defin}

It follows from item~5 of Definition~\ref{def:vector-pc} that $L_\alpha$ defined in \eqref{eq:tikhonov-pcls3} coincides with $\La$ in \eqref{eq:L}, for any $(\K,K, \phi) \in \Gamma(z)$  (indeed, the dualizing parameter function $f$ satisfies $f(\K,K, \phi,z) = + \infty$ whenever $(\K,K,\phi) \not\in \Gamma(z)$).
\\
Notice that $\La (\K,K, \phi;\lambda,\mu)$ coincides with
$\Fa(\K,K,\phi)$ in \eqref{eq:inv-probl-copt}, whenever $\W(\phi)=0$ and
$M(\K ,K)=0$ (indeed, $\W(\phi)=0$ and $M(\K ,K)=0$ imply $(\K,K, \phi)
\in \Gamma(0)$).

From item~8 of Definition~\ref{def:vector-pc}, it can be deduced that $Q(\lambda, \mu)$ is equal to the infimum of $\theta(z) - \rho(z,\lambda, \mu)$ over all $z$ in $(L^2(\Omega))^2$, with $\theta$ being the perturbation function and $\rho$ the
coupling function (see items~5 and~6).

Let $V_p := \inf\limits_{(\K,K,\phi) \in X} \widetilde{\Fa}(\K,K,\phi)$ and
$V_d := \sup\limits_{(\lambda,\mu) \in (L^2(\Omega))^2 \times \mathbb{R}_+}
Q(\lambda,\mu)$  be the optimal values of the primal and dual problems
respectively. It follows from the definitions of $f$ and $\rho$ that our scheme
has the weak duality property, i.e.
\begin{equation} \label{eq:weak-duality}
V_d \ \leq \ V_p \, .
\end{equation}

The difference between the values of $V_p$ and $V_d$ is known as the duality gap. The upcoming analysis of the augmented Lagrangian approach requires an exact penalty representation \cite{Melo2010, RockWets98} as its primary element, as it implies a zero duality gap, as shown in the following.

\begin{defin} \label{def:exact-penalty}
A vector $\bar{\lambda} \in (L^2(\Omega))^2$ is said to support an \textit{exact penalty representation} for the problem of minimizing $\widetilde{\Fa}$, if there exists $\bar{\mu} > 0$ such that for any  $\mu > \bar{\mu}$
\begin{equation} \label{eq:def-exact-penalty}
\theta(0) \ = \ Q(\bar{\lambda},\mu)   \quad\quad  {\rm and}
\quad\quad {\rm arg}\!\!\!\!\min_{(\K,K, \phi)} \, \widetilde{\Fa}(\K,K,\phi)
\ = \
{\rm arg}\!\!\!\!\min_{(\K,K, \phi)} \, \La (\K,K,\phi;\bar{\lambda},\mu) \, .
\end{equation}
Alternatively, such a vector $\bar{\lambda}$ is said to support an exact penalty representation for the problem of minimizing $\Fa $ under the constraints
$\W(\phi) = 0$ and $M(\K,K)=0$, see for example  \cite{Melo2010}.
\end{defin}


In order to prove the main result of this subsection, we shall first present some preliminary results concerning Definition~\ref{def:vector-pc}. 

\begin{lemma} \label{lemma:C-1-2}
The following assertions hold true:
\begin{itemize}
\item[i)] For any $(\lambda,\mu) \in (L^2(\Omega))^2\times \mathbb{R}_+$
the function $\rho(\cdot,\lambda,\mu)$ is upper semi-continuous at
$0$ and satisfies $\rho(0,\lambda,\mu) = 0$.
\item[ii)] For each $z \in (L^2(\Omega))^2$ and $\lambda \in (L^2(\Omega))^2$,
the function $\rho(z,\lambda, \cdot)$ is monotonically decreasing in $\mathbb{R}_+$.
\item[iii)] For every neighborhood $V \subset (L^2(\Omega))^2$ of $z=0$ and for every $(\lambda, \bar{\mu}) \in (L^2(\Omega))^2\times\mathbb{R}_+$,
it holds
\subitem{a)} $A^V_{\lambda,\bar{\mu}}(\mu) := \inf_{z \in V^C}
\{ \rho(z, \lambda, \bar{\mu}) - \rho(z, \lambda, \mu) \} > 0$,
$\forall \mu > \bar{\mu}$;
\subitem{b)} $\lim_{\mu \rightarrow \infty} A^V_{\lambda, \bar{\mu}}(\mu) = \infty$.
\end{itemize}
\end{lemma}
\begin{proof}
The continuity and monotonicity of $\rho$ as well as the fact that
$\rho(0, \lambda, \mu) = 0$ are direct consequences of the definition of the
coupling function $\rho$ (see Definition~\ref{def:vector-pc},
item~\ref{item-coupling}). This proves assertions~i) and~ii).
Assertion~iii) follows from the identity $\rho(z,\lambda,\bar{\mu})
- \rho(z,\lambda,\mu) = (\mu - \bar{\mu}) \norm{z}$.
\end{proof}

In the sequel, we demonstrate certain characteristics of the perturbation function $\theta$. We start by introducing the Fenchel-Moreau conjugate and biconjugate functions, as well as the abstract subgradient.

\begin{defin}[{\cite[Chapter~11]{RockWets98}}] \label{def:conjugate}
The Fenchel--Moreau conjugated and biconjugated functions of $\theta$ with respect to the coupling function $\rho$ are defined respectively by
$$
\theta^\rho(\lambda,\mu) = \sup\limits_{z\in (L^2(\Omega))^2}\{\rho(z,\lambda,\mu) - \theta(z)\}
\quad \mbox{ and } \quad
\theta^{\rho\rho}(z) = \sup\limits_{(\lambda,\mu) \in (L^2(\Omega))^2\times \mathbb{R_+}}
                       \{\rho(z,\lambda,\mu) - \theta^\rho(\lambda,\mu)\} \, .
$$
Moreover, given $\epsilon \ge 0$, an element $(\lambda,\mu) \in
(L^2(\Omega))^2\times \mathbb{R_+}$ is called $\epsilon$-abstract subgradient of
$\theta$ at $\bar{z}$ with respect to $\rho$ when
$\theta(z) - \rho(z,\lambda,\mu) \geq \theta(\bar{z}) - \rho(\bar{z},\lambda,\mu)- \varepsilon$, for all $z \in (L^2(\Omega))^2$.
The set of all $\varepsilon$-abstract subgradients of $\theta$ at $\bar{z}$ is called $\epsilon$-subdifferential of $g$ at $\bar{z}$ and is denoted by
$\partial_{\rho,\varepsilon} \theta(\bar{z})$.
\end{defin}

\begin{lemma} \label{lemma:six}
Consider the perturbation function $\theta$ as in Definition~\ref{def:vector-pc} 
\begin{itemize}
\item[i)] The perturbation function $\theta$ is lsc at $z=0$.
\item[ii)] Moreover, from the
definitions of $\theta^\rho$ and $\theta^{\rho\rho}$ it follows that \
$dom(\theta^\rho) \not= \emptyset$, \ $\theta^\rho(\lambda,\mu) = - Q(\lambda,\mu)$
\ and \ $\theta^{\rho\rho}(z) \leq \theta(z)$.
\item[iii)] The weak duality property is true, that is, $V_d \le V_p$.
\end{itemize}
\end{lemma}
\begin{proof}
Item i) and ii) are direct consequences of Definition~\ref{def:exact-penalty} and Definition~\ref{def:conjugate}, respectively. To prove item iii), first observe that $\rho(0,\lambda,\mu) = 0$ from Lemma~\ref{lemma:C-1-2}~i). Utilizing this fact in Definition~\ref{def:conjugate}, and item~ii), we can deduce that $\theta^{\rho\rho}(0) = V_d \leq \theta(0) = V_p$.
\end{proof}

We finally check the last supporting result before continuing to demonstrate the primary theorem of this subsection.
\begin{lemma} \label{lemma:seven}
Let $\rho$ and $\theta$ as in Definition~\ref{def:exact-penalty}. Furthermore, consider the $\varepsilon$-subgradient as in Definition~\ref{def:conjugate}. Then:
\begin{itemize}
\item[a)] If $(\lambda,\mu_0) \in \partial_{\rho,\varepsilon} \theta(0)$, then $(\lambda,\mu) \in \partial_{\rho,\varepsilon} \theta(0)$, for every $\mu \ge \mu_0$.

\item[b)] For all $\varepsilon > 0$ it holds $\partial_{\rho,\varepsilon} \theta(0)
\not= \emptyset$.

\item[c)] Let $(\bar\lambda,\bar\mu) \in \, dom(\theta^\rho)$ be given. For every
$\varepsilon > 0$ there exists a $\mu_0 = \mu_0(\varepsilon)$ such that
$(\bar\lambda,\mu) \in \ \partial_{\rho,\varepsilon} \theta(0)$, for all $\mu \ge \mu_0$.

\item[d)] Let $(\bar\lambda,\bar\mu) \in dom(\theta^\rho)$ be given. There exists
$$ \theta(z) \ \geq \ \theta(0) - \langle \bar{\lambda} , z \rangle
   - \hat{\mu} \norm{z}_{L^2} \, , \ \forall z \in (L^2(\Omega))^2 \, . 
$$
\end{itemize}

\end{lemma}
\begin{proof}
The assertion of item~a) is a consequence of the definition of the $\epsilon$-abstract
subgradient $\partial_{\rho,\varepsilon} \theta (\cdot)$, in addition to
Lemma~\ref{lemma:C-1-2} i) and ii).

Lemma~\ref{lemma:six}-i) and ii) and Lemma~\ref{lemma:C-1-2} guarantee that the requirements
of \cite[Theorem~5.2.1]{Melo09} (see also \cite{Melo2010}) are fulfilled, thus confirming
the validity of items~b) and c).

Finally, observe that item a) implies the existence of an element $(\bar\lambda,\bar\mu)$
in the domain of $\theta^\rho$. Define
$\hat{\mu} := \max \{ 1, \norm{\bar{\lambda}}\}$.
Thus, the estimate $-\langle \bar{\lambda}, z \rangle - \hat{\mu} \norm{z} \leq
(\norm{\bar{\lambda}} - \hat{\mu}) \, \norm{z} \leq 0$ holds true for all $z \in (L^2({\Omega}))^2$.

Given $\varepsilon > 0$, we can deduce from the Definition~\ref{def:conjugate} of $\varepsilon$-subgradient, the assertion in item~c) of this lemma, the definition of $\rho$, and Lemma~\ref{lemma:C-1-2}-i) that
\begin{align}\label{eq:13}
\theta(z) - \rho(z,\overline{\lambda}, \mu_\varepsilon)  \geq \theta(0) - \rho(0,\overline{\lambda}, \mu_\varepsilon) - \varepsilon = \theta(0) - \varepsilon \qquad \forall z \in (L^2(\Omega))^2.
\end{align}

From Lemma~\ref{lemma:C-1-2}-ii) and \eqref{eq:13}, it follows that $\theta(z) \geq \theta(0) - \langle \bar{\lambda} , z \rangle - \hat{\mu}\norm{z}$ for all $z \in (L^2(\Omega))^2$, which concludes the proof.
\end{proof}

Next, we prove the main result of this subsection: the zero-duality gap property for the proposed augmented Lagrangian approach. It should be noted that, as a result of  Lemma~\ref{lemma:six}-iii), this property is equivalent to $\theta^{\rho\rho}(0) = \theta(0)$.

\begin{theo} \label{theo:exact-penalty}
Let the main assumptions~\textbf{A1)} - \textbf{A3)} hold. Consider $\Fa$, $\W$ and $M$ be defined as above.
\begin{itemize} 
\item[(i)] There exists $\bar{\lambda} \in (L^2(\Omega))^2$ supporting an exact penalty representation in the sense of Definition~\ref{def:exact-penalty}. 

\item[(ii)] There is no duality gap, i.e., $\theta^{\rho\rho}(0) = V_d = V_p=\theta(0)$.
\end{itemize}
\end{theo}
\begin{proof}
From Assumption~\textbf{A3)}, we conclude that the set of primal solutions is not empty.
This implies that $\theta(0) < +\infty$ (see Definition~\ref{def:vector-pc}, item~6).
Additionally, Lemma~\ref{lemma:six}-i) guarantees that the perturbation function $\theta$
is lower semi continuous at $z=0$; consequently, $f(\K,K,\phi, \cdot)$ is also lower semi
continuous at $z=0$ for every $(\K,K,\phi) \in X$ (see Definition~\ref{def:vector-pc}, item~5).
Moreover, Lemma~\ref{lemma:C-1-2}-i) and ii) implies that the coupling function
$\rho(\cdot, \lambda, \mu)$ is upper semi continuous at $z=0$, monotonically decreasing
and satisfies $\rho(0, \lambda, \mu) =0$ (see Definition~\ref{def:vector-pc}, item-6).


On the other hand, Lemma~\ref{lemma:seven}~d) imply the existence of a $\hat{\mu} > 0$ satisfying
$$ \theta(z) \ \geq \ \theta(0) - \langle \bar{\lambda} , z \rangle
   - \hat{\mu} \norm{z}_{L^2} \, , \ \forall z \in (L^2(\Omega))^2 \, .
$$ 
Thus, arguing with \cite[Theorem~3]{Melo2010} we conclude that Assertion~$(i)$ holds true.

A consequence of Assertion~$(i)$ is that
$$
V_p = \theta^{\rho\rho}(0) =  Q(\bar{\lambda},\mu) \le \textstyle\sup_{(\lambda,\mu)} Q(\lambda,\mu) = \theta(0) = V_d \, .
$$
This inequality combined with the weak duality property \eqref{eq:weak-duality} allow us
to conclude that Assertion~$(ii)$ holds.
\end{proof}

\subsection{Convergence analysis for the double well (PCLS) approach}
\label{subsec:doble-well-convergence-analysis}

In this subsection, we investigate regularization properties of the minimizers
of the augmented Lagrangian approach proposed in this manuscript (see
Theorem~\ref{th:admissible-pcls}).
The proofs of these results are analogous to the ones presented in
\cite[Theorems~4, 5, 6]{Agnelli2018}; for the convenience of the reader,
only a sketch is presented.

\begin{theo} \label{th:admissible-pcls}

Under the assumptions of Theorem~\ref{theo:exact-penalty}, for any $\alpha > 0$,
we have:
\begin{itemize}
\item[i)] The functional $\widetilde{\Fa}$ attains minimizer on the set of admissible
functions satisfying the constraints $\W(\phi) =0$ and $M(\K,K)=0$.

\item[ii)] Problem \eqref{eq:inv-probl-copt} has a solution in the set of admissible
functions.

\item[iii)] Let $\bar{\lambda}_\alpha$ supporting an exact penalty representation
and $\mu_\alpha > \bar{\mu}_\alpha$ as in Definition~\ref{def:exact-penalty} (the
existence follows from Theorem~\ref{theo:exact-penalty}). 
Then, the augmented Lagrangian $\La(\cdot; \bar{\lambda}_\alpha, \mu_\alpha)$ attains
a minimizer on the set of admissible functions. 

\item[iv)] A solution of the problem \eqref{eq:inv-probl-copt} can be obtained by
solving the unconstrained optimization problem $\min_{(\K,K,\phi)}
\La(\K,K,\phi;\bar{\lambda}_\alpha,\mu_\alpha)$. 

\item[v)] [Convergence] \label{th:convergence} Assume that we have exact data,
i.e. $\delta = 0$ in \eqref{eq:noise}. 
For every $\alpha > 0$ let $(\K_\alpha, K_\alpha, \phi_\alpha$) be a corresponding
minimizer of $\La(\cdot; \bar{\lambda}_\alpha, \mu_{\alpha})$ in the set of
admissible functions (the existence is guaranteed by item~(iii) above). 
Then, for every sequence of positive numbers $\{\alpha_j\}$ converging to zero,
the corresponding sequence $\{\K_{\alpha_j}, K_{\alpha_j}, \phi_{\alpha_j}\}$ of
minima of $\mathcal{L}_{\alpha_j}(\cdot; \bar{\lambda}_{\alpha_j}, \mu_{\alpha_j})$
has a subsequence (which we denote by the same index)
$\{\K_{\alpha_j},K_{\alpha_j},\phi_{\alpha_j}\}$ that is strongly convergent in
$(L^2(\Omega))^2 \times L^p(\Omega)$, for $1 \leq p$. The limit of
$\{\K_{\alpha_j}, K_{\alpha_j}, \phi_{\alpha_j}\}$ is an admissible function.
Moreover, it is a solution of \eqref{eq:inv-probl-fpc} with $y^\delta = y$.

\item[vi)][Stability] \label{th:conv-stabil} 
Let $\alpha = \alpha(\delta)$ be a positive function such that $\lim_{\delta \to 0}
\alpha(\delta) = 0$ and $\lim_{\delta \to 0} \delta^2 / \alpha(\delta) = 0$.
Moreover, let $\{ \delta_j \}$ be a sequence of positive numbers converging to zero
and $\{ y^{\delta_j} \} \in Y$ be the corresponding noisy data satisfying
\eqref{eq:noise}. Then, there exists a subsequence, denoted again by
$\{ \delta_j \}$, and a sequence $\{ \alpha_j := \alpha(\delta_j) \}$
such that the corresponding minimizers $(\K_{\alpha_j}, K_{\alpha_j},\phi_{\alpha_j})$
of  $\Lak(\cdot; \bar{\lambda}_{\alpha_j}, \mu_{\alpha_j})$ converges in
$(L^2(\Omega))^2 \times L^p(\Omega)$, for $1\leq p$ to a solution
of \eqref{eq:inv-probl-fpc} with $y^\delta = y$.
\end{itemize}
\end{theo}

\noindent \textit{Sketch of the Proof:}
%
Assertion~$i)$ follows similarly to Lemma~3 in \cite{Agnelli2018},
where the continuity of $B(\cdot, \cdot)$ is given by Proposition~\ref{th:B-continuos}, presented in the appendix.
Assertion~$ii)$ follows immediately from Assertion~i) and the definition of
$\widetilde{\Fa}$ (see Definition~\ref{def:vector-pc}). Assertion~$iii)$
follows from Assertion~i) and~\eqref{eq:def-exact-penalty}.
Assertion~$iv)$ follows from Theorem~\ref{theo:exact-penalty} and
Assertions~$ii)$ and $iii)$. The proofs of Assertions~$v)$ and $vi)$ follow
the lines of \cite[Theorems~5 and~6]{Agnelli2018}, respectively.

\section{Numerical experiments}\label{sec:ap-bar-codes}

In this section, we investigate the performance and limitations of the augmented Lagrangian
PCLS regularization method introduced in Section~\ref{sec:level-set-formulation} for solving
the semi-blind deconvolution problem arrising in barcode decoding. To this end, we introduce
an ADMM-type algorithm \cite{BPCPE11} to compute a primal-dual solution for the augmented
Lagrangian
\begin{align} \label{eq:augmented-lagragian-special}
\tilde{L}_\alpha(\tilde{\sigma}, \sigma,\phi;\lambda,\mu)   := 
& \| B(K(\sigma),\phi) - y^\delta \|^2_{L^2(\Omega)}  \nonumber\\
&+ \alpha ( \beta_1 (\norm{\K(\tilde{\sigma}) - K_0}^2_{H^1(\Omega)} + \norm{K(\sigma) - K_0}^2_{H^1(\Omega)}) + \beta_2\ \| \phi \|^2_{L^{4n}} + \beta_3 | \phi |_\bv ) \notag\\
& + \langle \lambda_1 \, , M(\K(\tilde{\sigma}),K(\sigma)) \rangle_{L^2(\Omega)} + \langle \lambda_2 \, ,  \W(\phi) \rangle_{L^2(\Omega)} \nonumber \\
& + \mu ( \beta_4 \| M(\K(\tilde{\sigma}), K(\sigma)) \|_{L^2(\Omega)} + \| \W(\phi) \|_{L^2(\Omega)} ) \, .
\end{align}
Here, $\K(\tilde{\sigma}) := \K(\tilde{\sigma})(x) = \gamma(\tilde{\sigma} \sqrt{2\pi})^{-1}
e^{-(x^2/(2\tilde{\sigma}^2))} \chi_{\Omega}(x)$, $K(\sigma):=K(\sigma)(x) =
\gamma(\sigma \sqrt{2\pi})^{-1} e^{-(x^2/(2\sigma^2))}$ $\chi_{\Omega}(x)$,
$K_0 = K_0(x)$ is the \textit{a priori} in \eqref{eq:def-R-tilde}, and $\alpha$, $\beta_i$
are positive constants.

\subsection{An algorithm for the augmented Lagrangian PCLS regularization method}
\label{ssec:alg}

In what follows, we detail the primal-dual iterative algorithm employed in our numerical
experiments. Let $\sigma_0$, $\tilde{\sigma}_0 \in \mathbb{R}_+$, $\phi_0 \in L^2(\Omega)$,
$\lambda_0 = ((\lambda_1)_0,(\lambda_2)_0) \in (L^2(\Omega))^2$ and $\mu_0 \in \mathbb{R}_+$
be given.

\noindent\rule{\linewidth}{0.4pt} \\[-0.2ex]
{\bf Algorithm~1} \\[-1.5ex]
\noindent\rule{\linewidth}{0.4pt} \\[-4ex]
\begin{itemize}
\item[1] Update primal variables $(\tilde{\sigma}_k, \sigma_k, \phi_k)$ 
   \begin{itemize}
   \item[1.1] $ \phi_{k+1} \in  arg\min\limits_{\phi} \tilde{L}_\alpha(\tilde{\sigma}_k ,\sigma_k,\phi;\lambda_k,\mu_k)$
   \item[1.2] $(\tilde{\sigma}_{k+1}, \sigma_{k+1}) \in  arg\min\limits_{(\tilde{\sigma}, \sigma)}
   \tilde{L}_\alpha(\tilde{\sigma} , \sigma,\phi_{k+1};\lambda_{k},\mu_k)$
   \end{itemize}

\item[2] Update dual variables $(\lambda_k, \mu_k)$
   \begin{itemize}
   \item[2.1] $(\lambda_1)_{k+1} = (\lambda_1)_{k} + a_\lambda \,
   M(K(\sigma_{k+1}), \K(\tilde{\sigma}_{k+1}))$ \\[-1.5ex]
   %
   %
   %
   \item[2.2] $(\lambda_2)_{k+1} = (\lambda_2)_{k} +\frac{a_\lambda}{\|\W(\phi_{k+1})\|^2+\eps} \, \W(\phi_{k+1})$
   %
   \item[2.3] $\mu_{k+1} = \mu_k + \frac{a_\mu}{(\|\W(\phi_{k+1})\|
   + \beta_4 \|M( K(\sigma_{k+1}), \K(\tilde{\sigma}_{k+1}))\|)^2 + \eps} \,
   (\|\W(\phi_{k+1})\| + \beta_4 \|M( K(\sigma_{k+1}), \K(\tilde{\sigma}_{k+1}))\|)$
   \end{itemize}
\end{itemize}
\mbox{} \\[-4ex]
\noindent\rule{\linewidth}{0.4pt} \\[-1.5ex]

\noindent
where $a_\lambda$, $a_\mu$, and $\eps$ are positive constants ($\eps$ small). Algorithm~1 is terminated when the
constraints $\|\W(\phi_k)\|=0$ and $\|M(\K(\tilde{\sigma}_k), K(\sigma_k)) \|= 0$ are
simultaneously satisfied within a prescribed accuracy.

\begin{remark}[Update of the primal and dual variables in Algorithm~1] \mbox{}
\\
$\bullet$ Update of the primal variables $(\tilde{\sigma}, \sigma, \phi)$: To solve the
optimization problems in Step~1 of Algorithm~1, we proceed as follows:%
\footnote{Both Steps~1.1 and~1.2 are implemented using the Python library SciPy \cite{VirSciPy20}
with its default parameters.}

In Step 1.1, the BFGS method \cite{BLNZ95} is employed to solve the minimization problem with
respect to $\phi$. We use the limited-memory variant, referred to as \textit{Limited-memory BFGS}
(L-BFGS-B),%
\footnote{See \href{https://docs.scipy.org/doc/scipy/reference/generated/scipy.optimize.minimize.html}
{https://docs.scipy.org/doc/scipy/reference/generated/scipy.optimize.minimize.html}}
which exploits the a priori information that the barcode values $\phi$ are constrained to lie
in $[0,1]$.

In Step~1.2, the CG method \cite{HS52} is used to solve the minimization problem with respect
to $(\tilde{\sigma}, \sigma)$.%
\footnote{See \href{https://docs.scipy.org/doc/scipy/reference/generated/scipy.sparse.linalg.cg.html}
{https://docs.scipy.org/doc/scipy/reference/generated/scipy.sparse.linalg.cg.html}}
\\
$\bullet$ Update of the dual variables $(\lambda_1, \lambda_2, \mu)$: 

In Step~2.1, we use a gradient method with contant stepsize $a_\lambda$ to update $\lambda_1$.

In Steps~2.2 and~2.3, to avoid slow convergence (close to the solution) of the gradient method,
we use a variant of the normalized gradient method  \cite{Shor12} (see also \cite[Chapter~3.3.4]{Poly10}
to update $\lambda_2$
and $\mu$ respectively. A small positive constant $\eps$ is added in the denominator to enhance
stability and to limit the step size when the gradient becomes small.
\end{remark}

\subsection{Numerical setup for the barcode recovering problem}

The ground-truth barcode $\hat{u}$ used in our numerical experiments is depicted in
Figure~\ref{fig:Data01} (TOP). This barcode is the same as that employed in the
numerical simulations presented in \cite{E04}.
Following the ansatzt in~\eqref{eq:def-pcls}, $\hat{u}$ is modeled by a function
$\hat{\phi}: [-1, 1] \rightarrow \mathbb{R}$, where $\hat{\phi}(x)=0$ corresponds to the
black bars and $\hat{\phi}(x)=1$ corresponds to the white spaces between the bars,
for all $x \in [-1, 1]$; see Figure~\ref{fig:Data01} (BOTTOM).

As explained in the Introduction, the barcode recovery problem consists in
simultaneously identifying the pair $(\hat u, \hat\sigma)$ from noisy data.
In each experiment, the barcode $\hat{u}$ is corrupted by a convolution
with the PSF kernel $K$ defined in \eqref{def:PSF}, using a known standard deviation
$\hat\sigma$, followed by the addition of uniformly distributed noise satisfying
\eqref{eq:noise}. To illustrate the numerical setup of this inverse problem
we show in Figure~\ref{fig:Data01} the scenario corresponding to $\hat\sigma=0.026$
and $\delta = 0.5\%$.

Three distinct variances levels ($\hat{\sigma}^2$) are considered in our experiments,
namely $\numparam{\hat\sigma=0.024}$, $\numparam{\hat\sigma=0.026}$, and
$\numparam{\hat\sigma=0.028}$.
For each of these values, two levels of relative noise are examined,
namely \numparam{$\delta = 0.5\%$} and \numparam{$\delta = 5.0\%$}
(an additional experiment with $\delta = 10\%$ is also considered;
see Figures~\ref{fig:experiment_sigma028noise10}
and~\ref{fig:error_reconstructions_noise10}).

\begin{remark}[On the physical interpretation of the experiments]
There is no universal relation of the form $d = f(\sigma)$ that
directly links the scanner-barcode distance $d$ to the standard deviation
$\sigma$ without introducing an additional physical model or a
calibration procedure.

The model adopted in \eqref{def:PSF} relies on the qualitative assumption 
$\sigma = \kappa d L$, where $L$ denotes the barcode width and $\kappa$
is a system-dependent constant (physically, the constant $\kappa$
is directly related to the angular divergence of the laser beam).
Datasheets of commercial laser modules and low-cost 1D barcode
scanners (e.g., Zebra Technologies / Symbol and Datalogic devices
commonly used in supermarkets) indicate that the total angular
divergence of the laser beam typically ranges between 0.004
and 0.015 radians. Consequently, calibrated values of $\kappa$ are
generally expected to lie on the order of $10^{-3}$.

In our setting, the computational domain is normalized to $\Omega = (-1,1)$.
We assume a typical barcode width of $L = 3$cm (e.g., an EAN-13 code, which
is widely used in retail environments), and further assume that representative
values of $\kappa$ for commercial (low-cost) barcode scanners (e.g., supermarket
scanners) lie within the range: $\kappa \in (0.002, 0.003)$.
Therefore, for low-cost scanners, the range $\sigma \in (0.024, 0.028)$
considered in our experiments corresponds to a physical scanner-barcode distance
of approximately $d \in (24\text{cm}, 42\text{cm})$.
\end{remark}

\begin{figure}[!t]
\centering    
\includegraphics[width=\textwidth]{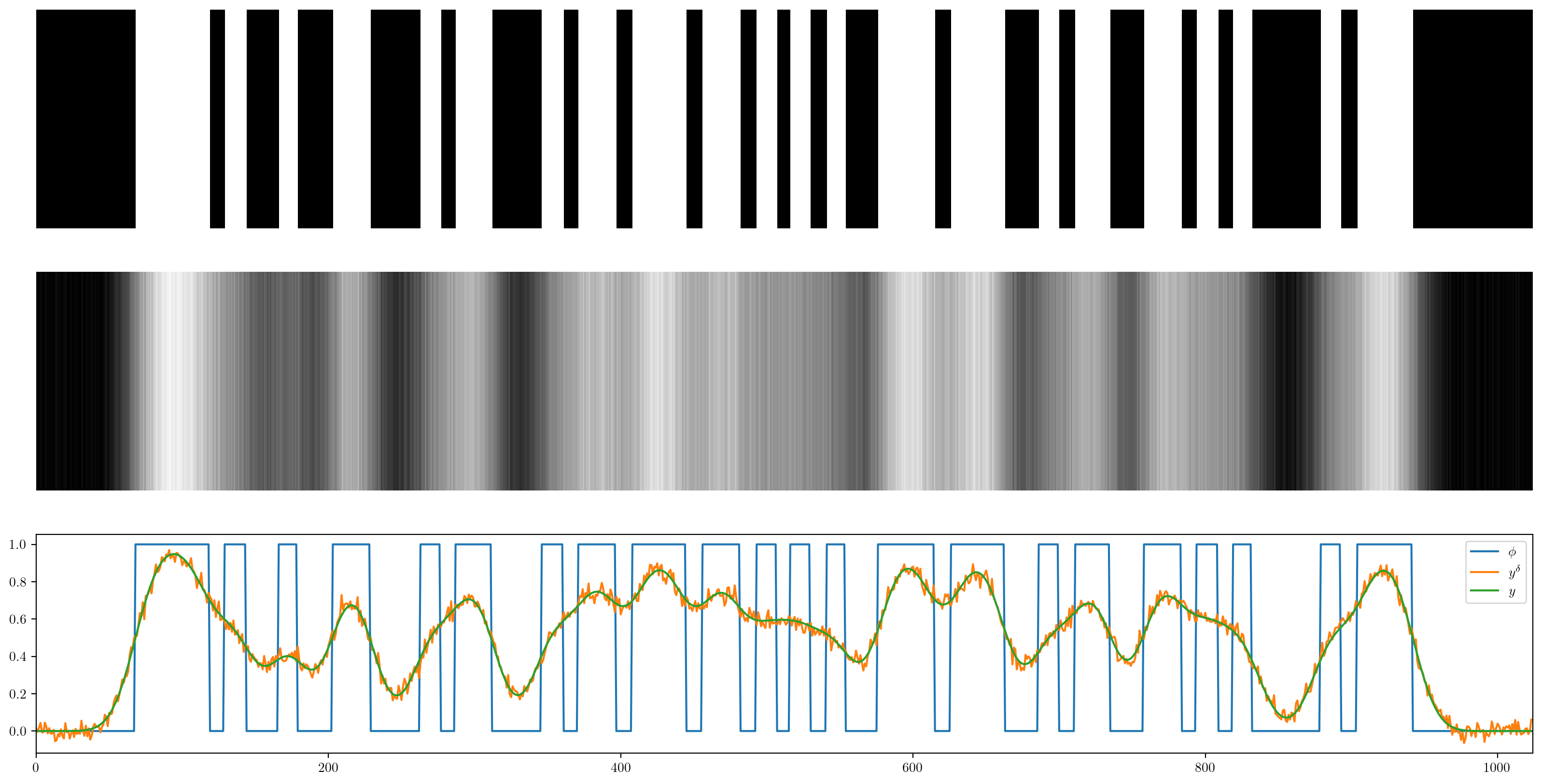}
\caption{\small Numerical setup for the barcode recovery problem (\numparam{$\hat\sigma=0.026$}
and \numparam{$\delta=0.5\%$}): (TOP) ground-truth barcode $\hat u$; (CENTER) observed signal
obtained from exact barcode by convolution with a Gaussian kernel followed by the
addition of noise; (BOTTOM) functions $\hat\phi$ (modeling the barcode $\hat u$),
$y$ (exact data obtained by convolving $\hat\phi$ with the PSF kernel) and
$y^\delta$ (noisy data).}
\label{fig:Data01}
\end{figure}
\FloatBarrier

\subsection{Implementation aspects of Algorithm~1}

The function $\hat\phi$, modeling the ground-truth barcode $\hat{u}$, is
discretized on a uniform mesh of \numparam{1024} equally spaced points on the
interval $[-1,1]$. This same mesh is used to represent the regularized solutions
in all experiments conducted in this section.

The double-well $\W(\phi): = \phi^n(\phi-1)^n$ in \eqref{eq:inv-probl-fpc} is chosen
with \numparam{$n=2$}. In the simulations, the regularization parameter is set to
\numparam{$\alpha= 10^{-8}$}, while the scaling factors $\beta_j$ in
\eqref{eq:augmented-lagragian-special} are chosen as
\numparam{$\beta_1 = \beta_3= 10^{-8}$}, \numparam{$\beta_2 = 1$},
\numparam{$\beta_4 = 0.02$}, and the constant $\eps$ is set to \numparam{$10^{-10}$}.

\begin{figure}[!t]
    \centering  \includegraphics[width=\textwidth]{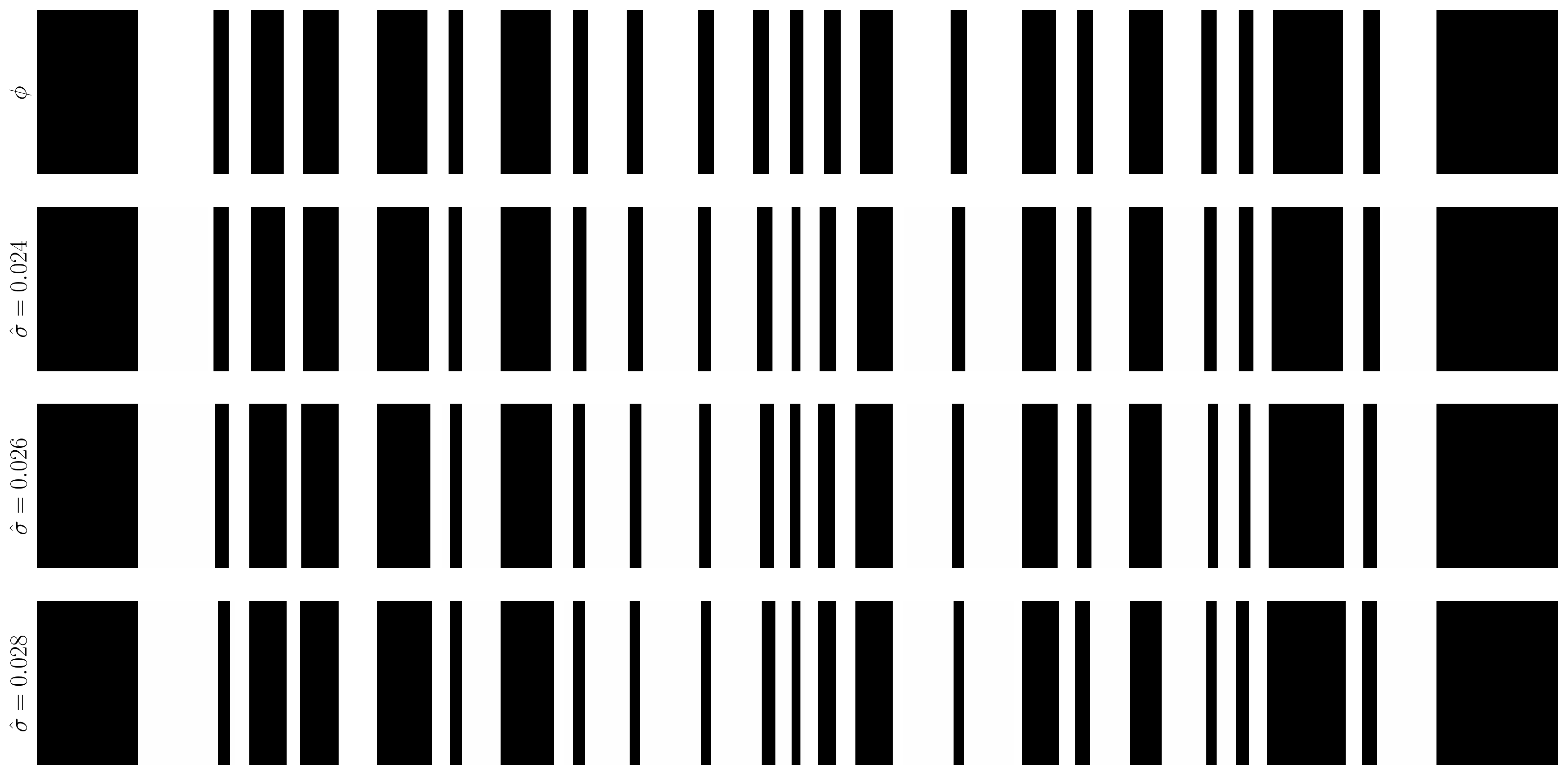}
    \caption{Reconstructed barcodes for \numparam{$\delta=0.5\%$} and different values of $\hat\sigma$.}
\label{fig:error_reconstructions_noise005}
\end{figure}
\begin{figure}[!t]
    \centering \includegraphics[width=\textwidth]{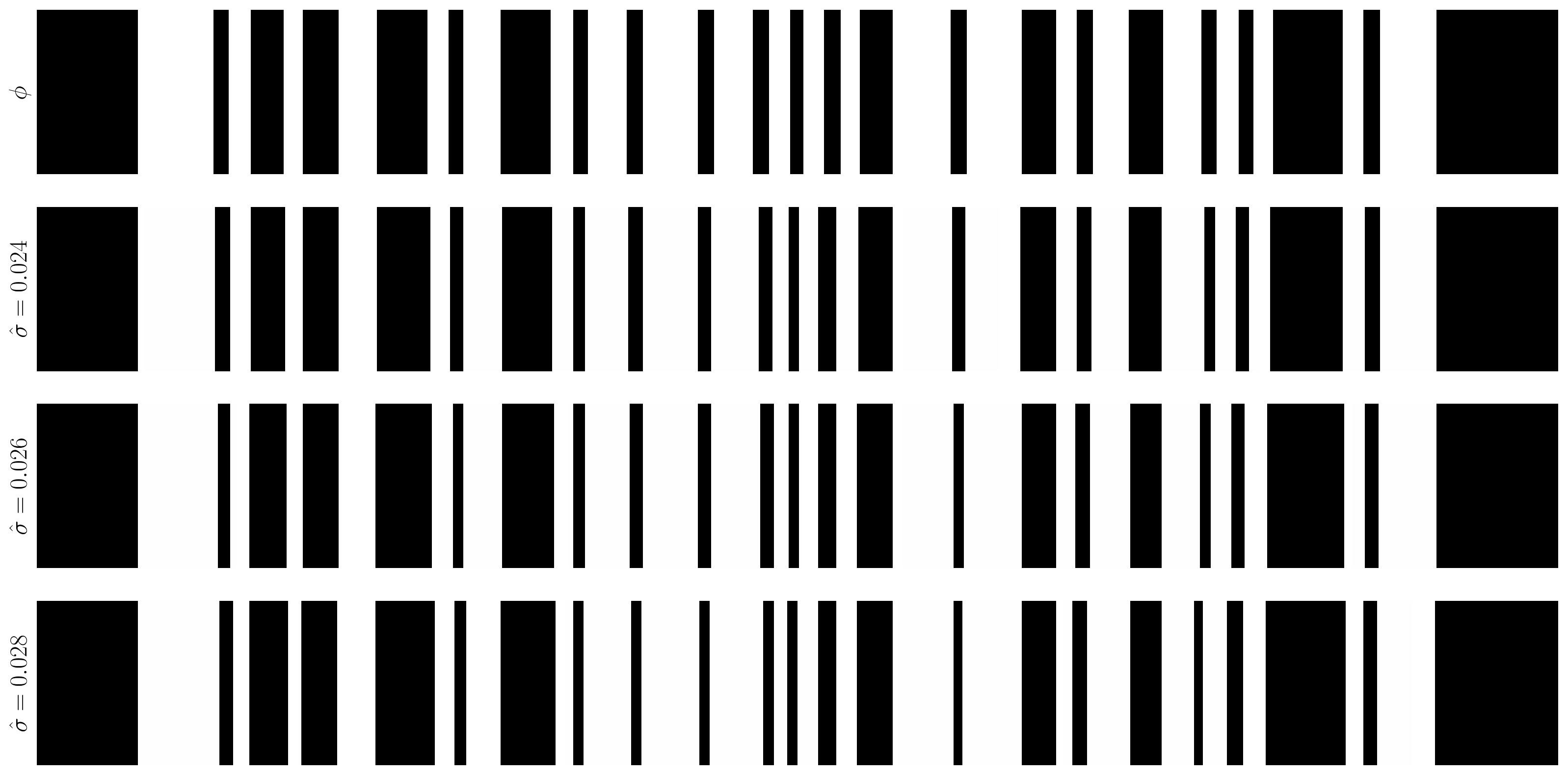}
    \caption{Reconstructed barcodes for \numparam{$\delta=5\%$} and different values of $\hat\sigma$.}
    \label{fig:error_reconstructions_noise05}
\end{figure}
\FloatBarrier

As initial guess $\phi_0$ for Algorithm~1 we set \numparam{$\phi_0(x) \equiv 0.1$},
i.e., a constant function representing an almost completely black barcode.
We assume that the amplitude of the kernel $K$ is $\gamma = 1$. The initial guess
$\sigma_0 \in \mathbb{R}_+$ for the standard deviation of the kernel
$K_0 = \overline{K}(\sigma_0)$ is set to \numparam{$\sigma_0 = 0.02$}.
The initial guess for the standard deviation associated with the slack variable
$\mathcal{K}$ is set to \numparam{$\tilde{\sigma}_0 = 10^{-3}$}, which corresponds
to a low-intensity blurring.

As initial guess for the dual variables ($\lambda$, $\mu$) of the augmented Lagrangian
we use the values \numparam{$\mu_0=0.04$} and \numparam{$\lambda_0(x)=(0,0)$}. The constants
\numparam{$a_\mu = a_\lambda=0.0004$} are used in the implementation of Step~2.

As stopping criteria, Algorithm~1 is terminated when both conditions
$$\|\W(\phi_k)\|_{L^1(\Omega)}\leq 10^{-5}\ \  {\rm and} \ \ 
\|\M(K(\sigma_{k}), \K(\widetilde{\sigma}_{k}))\|_{L^1(\Omega)}\leq 10^{-5}$$
are simultaneously satisfied.

\subsection{Numerical reconstructions}

In what follows, we discuss the performance of Algorithm~1 in Section~\ref{ssec:alg}
for the semi-blind deconvolution problem of simultaneously identifying the barcode $\hat{u}$ and
the variance $\hat{\sigma}^2$.
We test the method in three distinct scenarios corresponding to blurred barcodes with standard deviations
\numparam{$\hat{\sigma}= 0.024$}, \numparam{$\hat{\sigma}= 0.026$} and \numparam{$\hat{\sigma}=0.028$}.
In each scenario, we consider noisy measurements $y^\delta$, contaminated with noise levels
\numparam{$\delta=0.5\%$} and \numparam{$\delta=5\%$}.

The reconstruction results shown in Figure~\ref{fig:error_reconstructions_noise005}
(for $\delta = 0.5\%$) and Figure~\ref{fig:error_reconstructions_noise05} (for $\delta = 5\%$)
demonstrate the potential of Algorithm~1 to recover both the widths and the positions of
the bars and gaps in the ground-truth barcode $\hat u$, across all considered scenarios
of blur $\hat\sigma$, which ranges over the interval $\hat\sigma \in [$\numparam{$0.024$}, \numparam{$0.028$}$]$.
To complement the discussion of the results shown in these two figures,
Figure~\ref{fig:ErrorRelativeAndSigmaEvol} presents two plots that are relevant for
assessing the efficiency of Algorithm~1:
(i) Figure~\ref{fig:RelativeError} shows the evolution of the {\bf relative iteration error}
$$\|\phi_k - \hat\phi \|_{L^1} / \|\hat\phi\|_{L^1}\,, $$
for the six experiments presented in Figures~\ref{fig:error_reconstructions_noise005}
and~\ref{fig:error_reconstructions_noise05} (corresponding to three different values of
the variance $\hat\sigma^2$ and two distinct noise levels $\delta$);
(ii) Figure~\ref{fig:SigmaEvolution} displays the evolution of the iterated values
$\sigma_k$ corresponding to the same experiments examined in Figure~\ref{fig:RelativeError}.

It is important to note that Algorithm~1 is capable of producing reasonable approximations
of the ground-truth barcode after only a few iterations. To illustrate this finding, we
detail the barcode sequences generated by Algorithm~1 in two distinct scenarios:
(i) Figure~\ref{fig:Evolution_sigma024noise005}, corresponding to \numparam{$\hat{\sigma}= 0.024$}
and \numparam{$\delta = 0.5\%$}; (ii) Figure~\ref{fig:Evolution_sigma028noise05},
corresponding to \numparam{$\hat{\sigma}= 0.028$} and \numparam{$\delta = 5\%$}.
In both experiments, the algorithm yields a sufficiently accurate approximation
of the ground-truth barcode after only three iterations.

\begin{figure}[!t]
    \centering
    \begin{subfigure}{.50\textwidth}
      \centering \includegraphics[width=\linewidth]{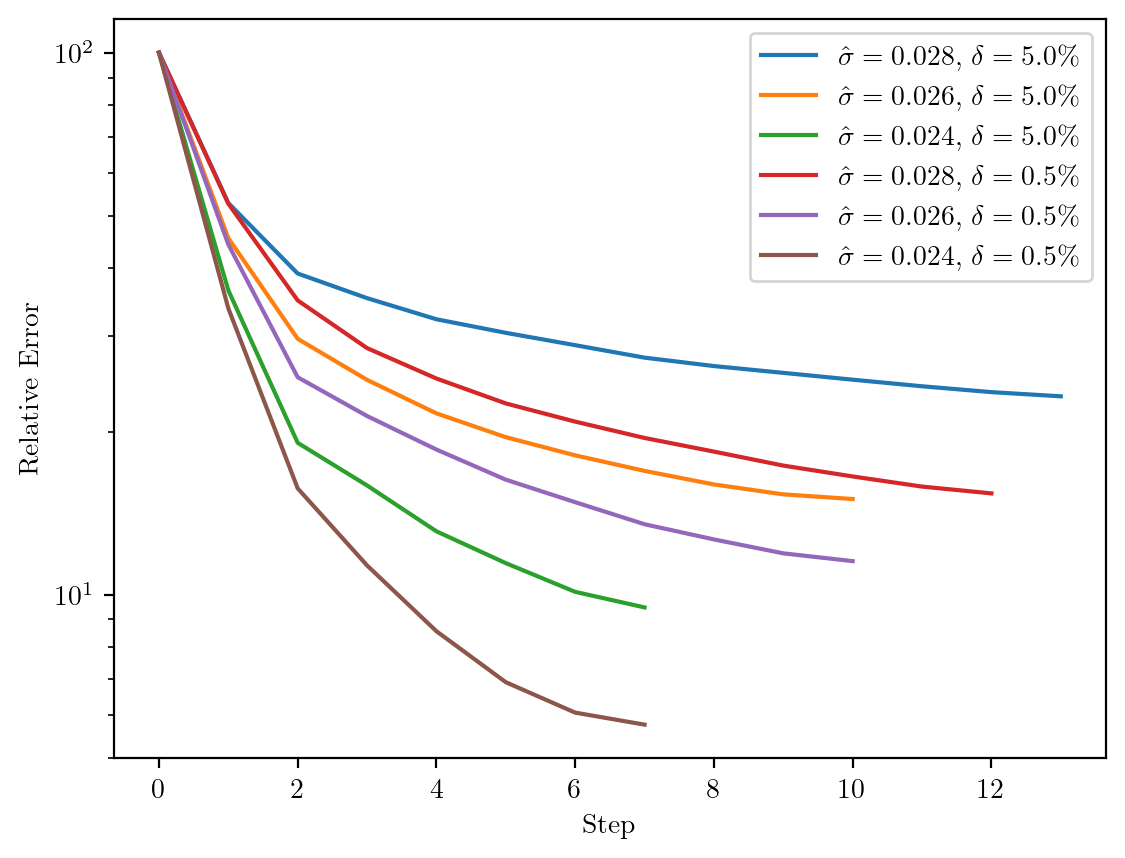}
      \caption{}
      \label{fig:RelativeError}
    \end{subfigure}%
    \begin{subfigure}{.50\textwidth}
      \centering \includegraphics[width=\linewidth]{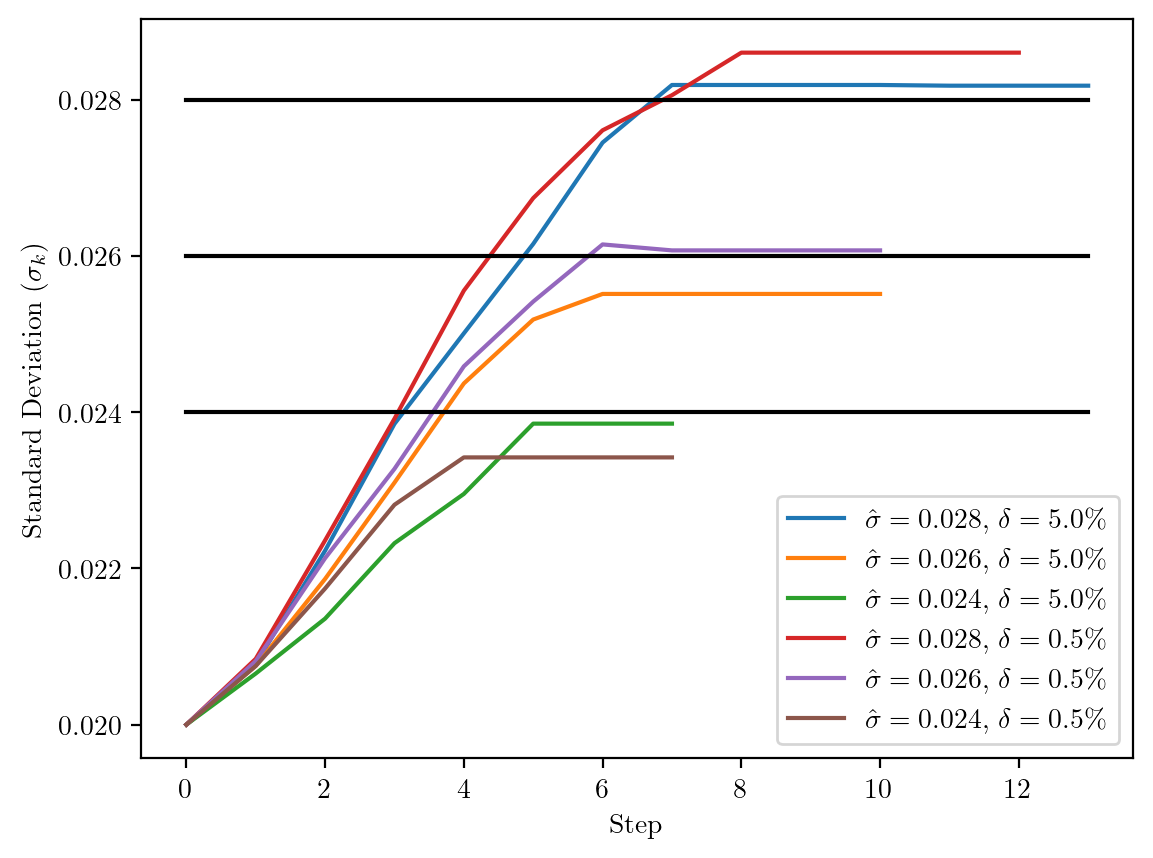}
      \caption{}
      \label{fig:SigmaEvolution}
    \end{subfigure}
    \caption{(a) Relative error ($L^1$-norm) and (b) evolution of reconstructed $\sigma_k$ for all experiments.}
    \label{fig:ErrorRelativeAndSigmaEvol}
\end{figure}
\FloatBarrier

\begin{figure}[!t]
    \centering \includegraphics[width=\textwidth]{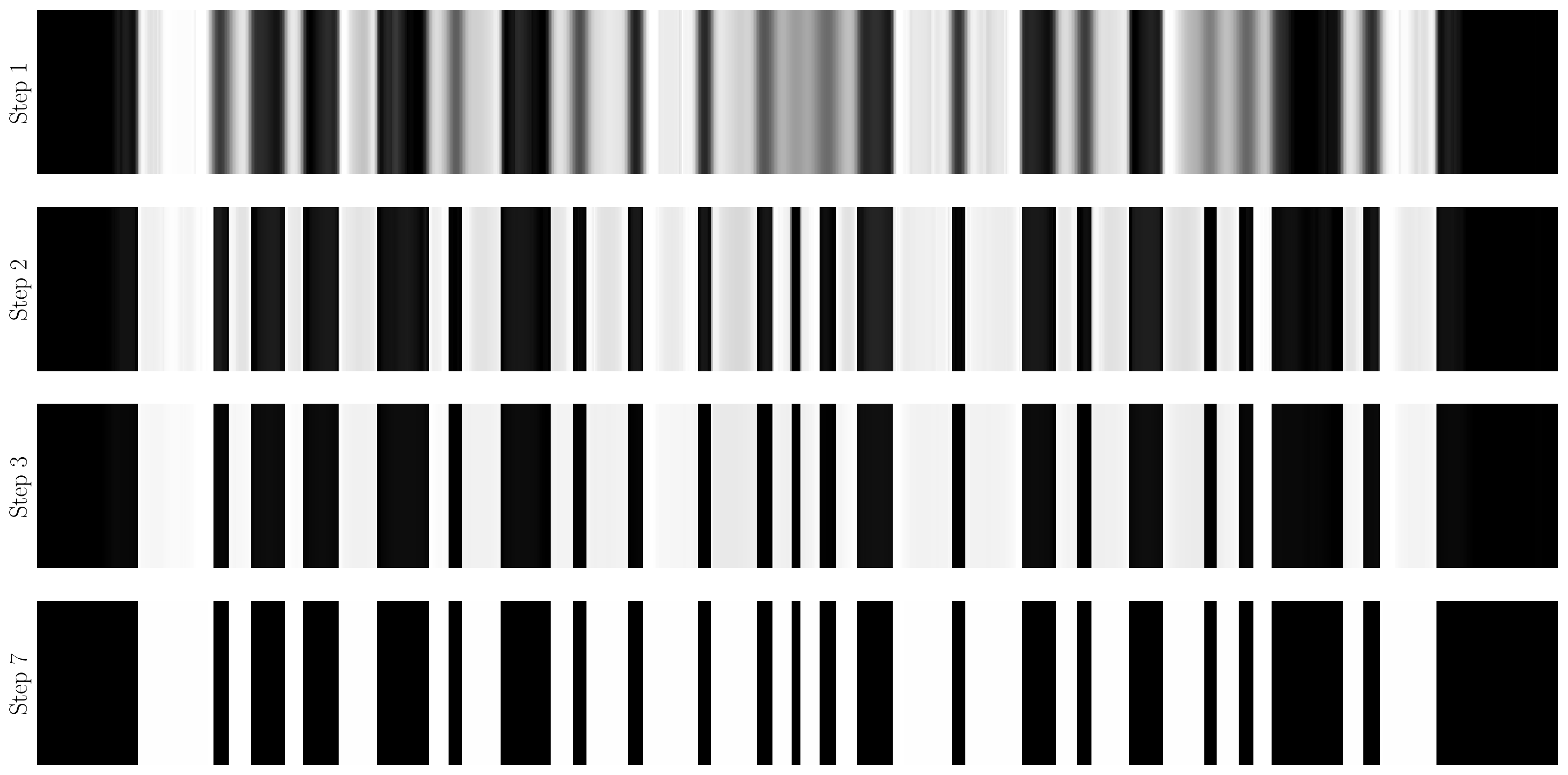}
    \caption{Evolution of reconstructed barcode for $\hat\sigma=0.024$ and $\delta=0.5\%$}
    \label{fig:Evolution_sigma024noise005}
\end{figure}
\begin{figure}[!t]
    \centering\includegraphics[width=\textwidth]{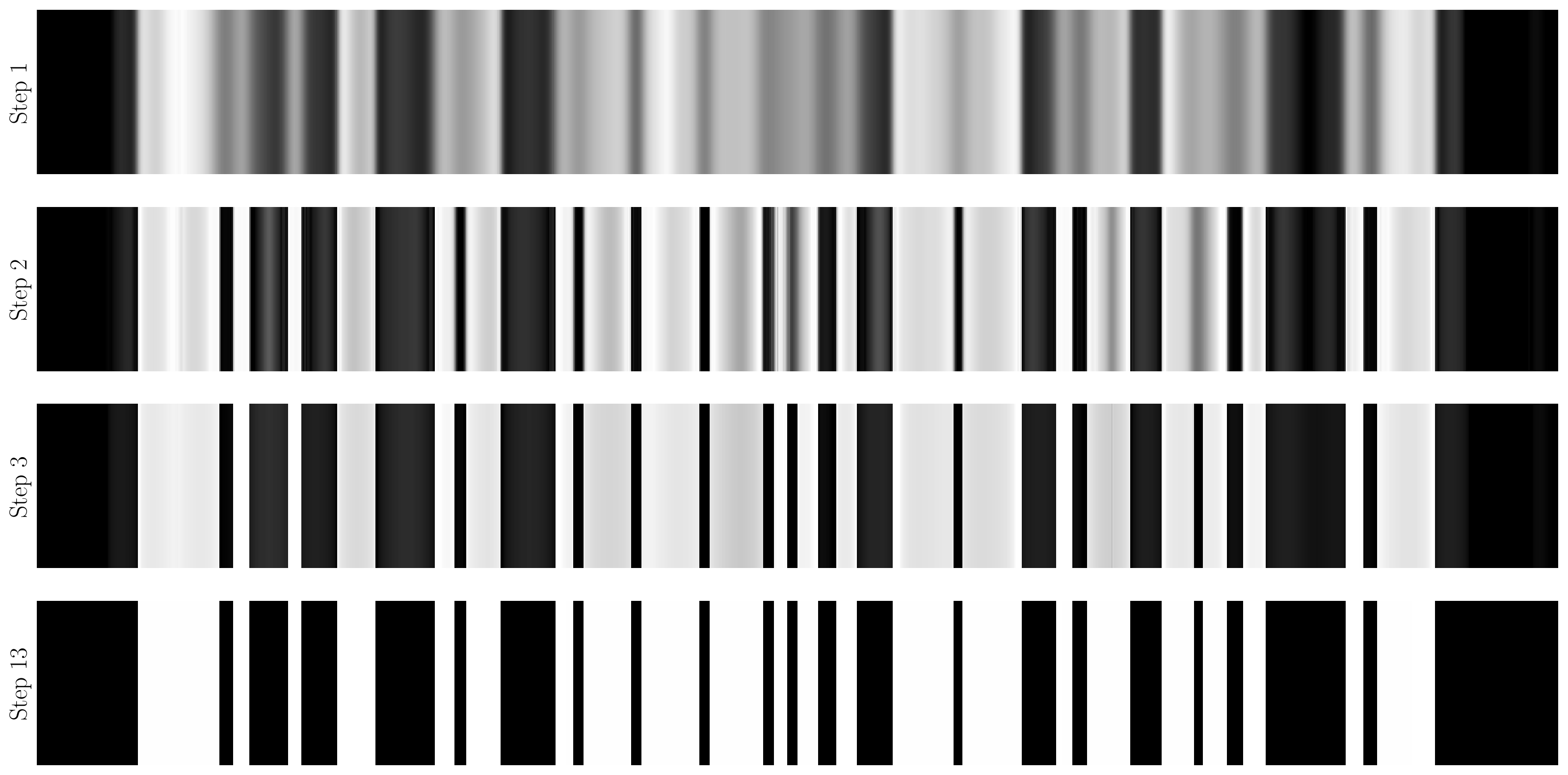}
    \caption{Evolution of reconstructed barcode for $\hat\sigma=0.028$ and $\delta=5\%$}
    \label{fig:Evolution_sigma028noise05}
\end{figure}
\FloatBarrier

In Figure~\ref{fig:experiment_sigma028noise05}, we present in detail the reconstruction results
obtained for the scenario \numparam{$\hat{\sigma}=0.028$} (large variance) and \numparam{$\delta=5\%$}
(high noise level).%
\footnote{The bottom plot in Figure~\ref{fig:Evolution_sigma028noise05} corresponds to the third
plot from the top in Figure~\ref{fig:experiment_sigma028noise05}.}
This represents the most unfavorable scenario for which Algorithm~1 was still able to
produce reasonably reliable reconstructions, without loss of bars, significant positional
shifts, or the generation of spurious bars.

We conclude this section by presenting a discussion of the main findings obtained from the
numerical experiments reported in this manuscript.
\begin{itemize}[leftmargin=*]
\item \textbf{Influence of variance values in the reconstruction:}
Different values of the blur $\sigma$ have a strong impact on the reconstruction quality.
As the standard deviation increases from \numparam{$\hat\sigma= 0.024$} to
\numparam{$\hat\sigma=0.028$} the quality of the reconstructed barcode progressively degrades,
meaning that the width and position of the reconstructed barcode bars become increasingly inaccurate. \\
For larger values of $\hat\sigma$, the thin black bars in $\hat u$ tend to become progressively
thinner and spatially displaced in the reconstructed barcode; as expected, this effect
becomes more pronounced for higher levels of noise in the data
(see Figures~\ref{fig:error_reconstructions_noise005} and~\ref{fig:error_reconstructions_noise05}
and compare the distinct reconstructions corresponding to the same values of $\hat\sigma$).

\begin{figure}[!t]
    \centering 
    \includegraphics[width=\textwidth]{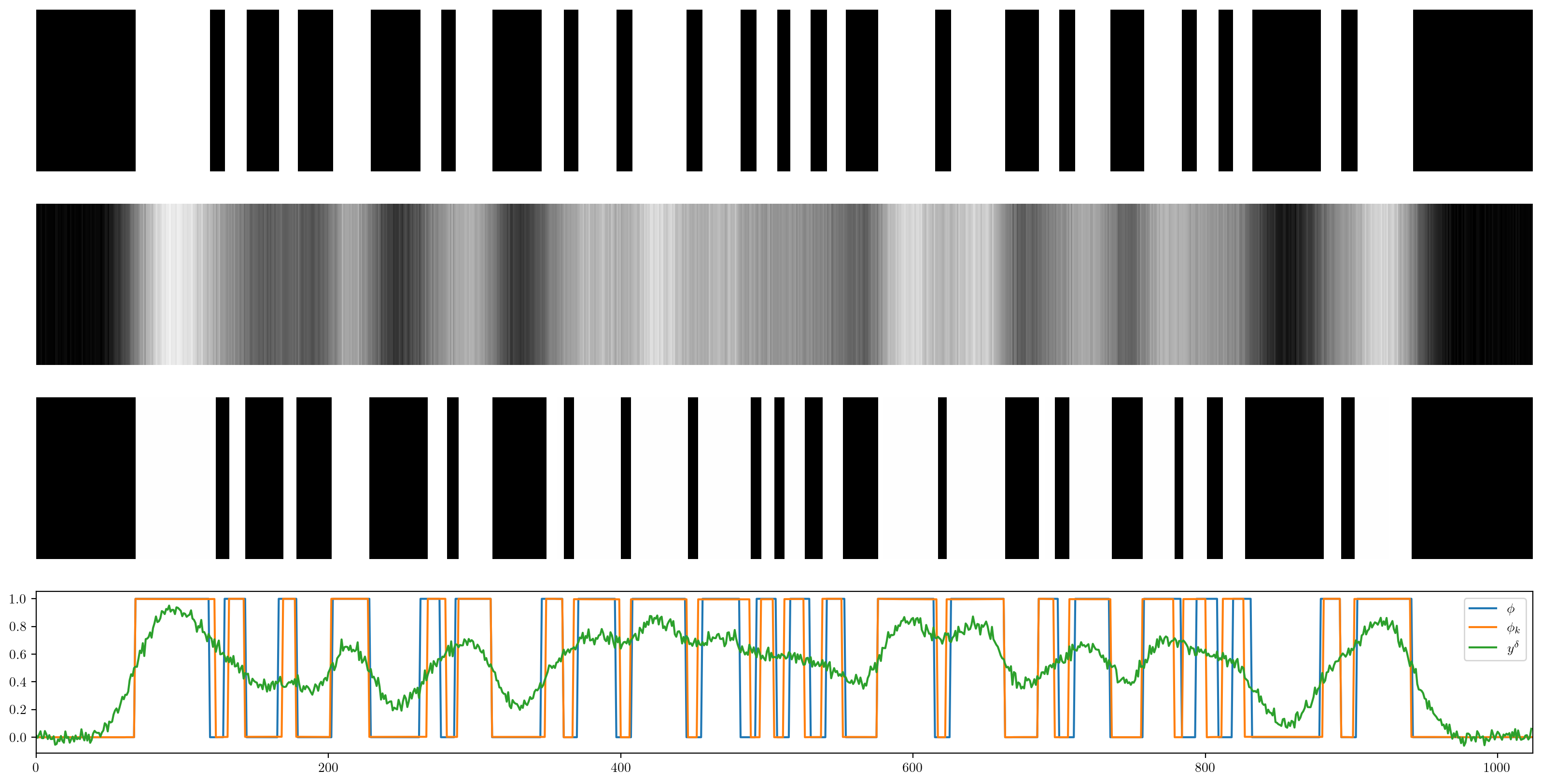}
    \caption{Scenario $\hat\sigma=0.028$ and $\delta=5\%$. From top to bottom: (1) ground-truth
    barcode $\hat u$; (2) noisy data $y^\delta$; (3) reconstructed barcode; (4) functions $\hat\phi$,
    $\phi_k$ for \numparam{$k = 7$}, and $y^\delta$}.
    \label{fig:experiment_sigma028noise05}
\end{figure}

From Figure~\ref{fig:RelativeError}, it is clear that the relative iteration error
$\|\phi_k - \hat\phi \|_{L^1} / \|\hat\phi\|_{L^1}$ attains smaller values at the end of
the iteration (when the stopping criterion is met) for smaller values of $\hat\sigma$,
indicating a higher accuracy of the reconstructed barcode.
In Figure~\ref{fig:RelativeError}, compare the BLUE and RED curves (corresponding to
the standard deviation $\hat\sigma=0.028$ for noise levels  $\delta =5.0\%$ and $\delta=0.5\%$, respectively)
with the GREEN and BROWN curves (corresponding to the standard deviation $\hat\sigma=0.024$ for noise levels 
$\delta =5.0\%$ and $\delta=0.5\%$, respectively).

For values larger than $\sigma = 0.028$, Algorithm~1 fails to produce reliable
reconstructions (independent of the noise level). In this regime, the reconstructed
barcode tends to lose bars (particularly the thinner ones) or to produce spurious bars
that are not present in the ground-truth barcode.

\item \textbf{Influence of the noise level in the reconstruction:}
Increasing the noise intensity~$\delta$ has a less pronounced effect on the quality of
the reconstructed barcode than increasing the parameter~$\sigma$. This behavior
is clearly illustrated in Figure~\ref{fig:RelativeError}.

Our experiments indicate that when noise level exceeds \numparam{$10\%$} and the
standard deviation exceeds \numparam{$\sigma=0.026$}, Algorithm~1 fails, leading to the
disappearance of bars and the emergence of spurious thin bars.
For a concrete example we present in Figure~\ref{fig:experiment_sigma028noise10} the
details of the barcode reconstruction in the scenario \numparam{$\hat{\sigma} = 0.028$}
and \numparam{$\delta = 10\%$}.
Additionally, in Figure~\ref{fig:error_reconstructions_noise10}, we show the reconstructed
barcodes for \numparam{$\delta = 10\%$} and different values of exact variance $\hat\sigma^2$.

\item \textbf{Choice of the initial guess~$\sigma_0$:} Experimental results indicate
that choosing a small value for the initial guess $\sigma_0$ leads to improved
performance of the iterative algorithm. When $\sigma_0$ is large, the iterates
$\phi_k$ tend to lose thin bars or to develop new spurious thin bars as the algorithm
progresses. This behavior has also been reported in~\cite{E04}.

\begin{figure}[!t]
    \centering 
    \includegraphics[width=\textwidth]{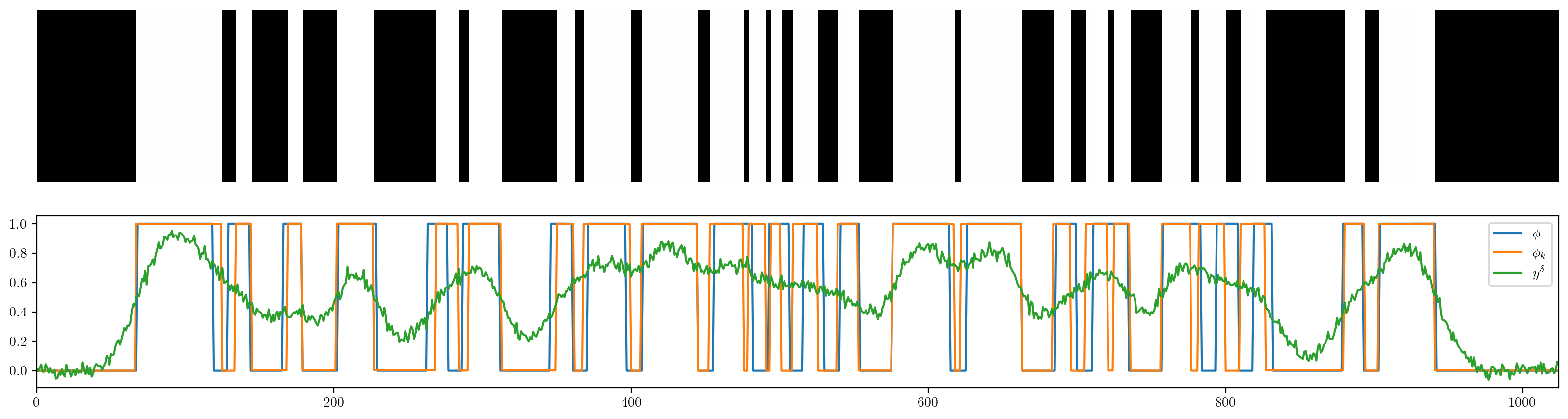}
    \caption{Scenario $\hat\sigma=0.028$ and $\delta=10\%$. (TOP) Reconstructed barcode;
    (BOTTOM) functions $\hat\phi$, $\phi_k$ for $\numparam{k=7}$ and noisy data $y^\delta$.}
    \label{fig:experiment_sigma028noise10}
\end{figure}
\FloatBarrier

\item \textbf{Fast convergence property}:
Algorithm~1 is capable of successfully the locations and widths of the ground-truth
bars and gaps at early stages of the iteration, typically between the second and third
steps. Subsequent steps are mainly devoted to enforcing that the reconstructed values
remain within the set~$\{ 0, 1 \}$.

\begin{figure}[!t]
    \centering  \includegraphics[width=\textwidth]{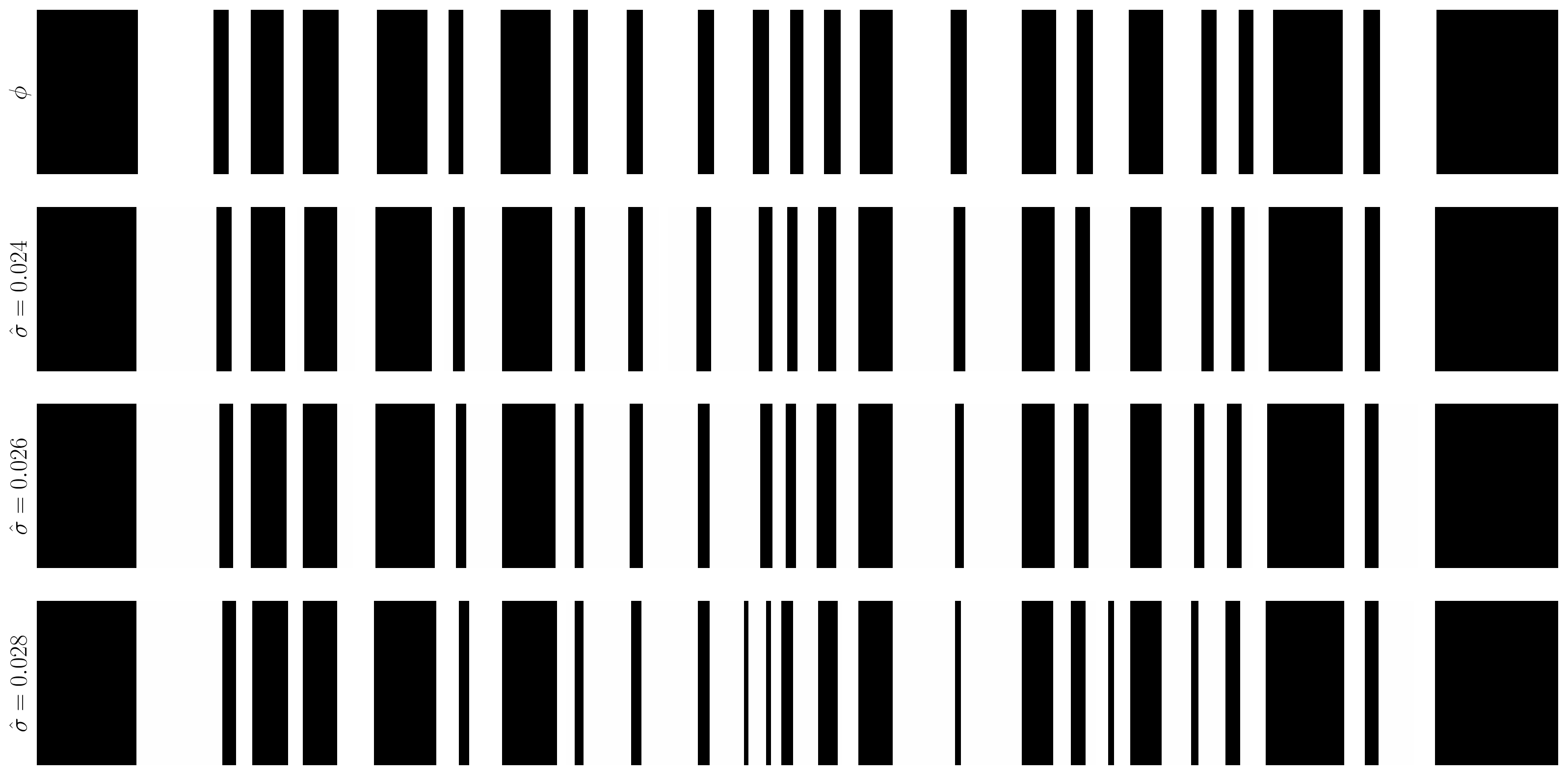}
    \caption{Barcode reconstructions for $\delta=10\%$ and different values of variance values.}
    \label{fig:error_reconstructions_noise10}
\end{figure}
\FloatBarrier

For completeness, we illustrate this rapid convergence by examining the sequences
generated by the method in two representative scenarios: (i) Figure~\ref{fig:Evolution_sigma024noise005},
corresponding to \numparam{$\hat{\sigma} = 0.024$} and \numparam{$\delta = 0.5\%$};
(ii) Figure~\ref{fig:Evolution_sigma028noise05}, corresponding to
\numparam{$\hat{\sigma} = 0.028$} and \numparam{$\delta = 5\%$}.

In the vast majority of practical applications under consideration, the primary interest
lies in the accurate identification of a binary barcode (the precise value of the variance
is typically irrelevant to the end user). A fast and efficient strategy for accomplishing
this task can be devised through the following two-step procedure:

\ \ \ \ \ (I) Perform only $k=3$ iterations of Algorithm~1, obtaining the approximation $\phi_3$
(generally \\ \mbox{}\ \ \ \ \ non-binary);

\ \ \ \ \ (II) Apply a thresholding step to $\phi_3$ in order to obtain a binary function
$\tilde\phi_3$ (e.g., $\tilde\phi_3 = 1$ if \\ \mbox{}\ \ \ \ \ $\phi_3 > 0.5$ and
$\tilde\phi_3 = 0$ otherwise).

The function $\tilde\phi_3$ is then returned as an approximation of the unknown barcode.
This strategy exploits a characteristic behavior of Algorithm~1, namely its ability to produce
reasonable (though not yet binary) approximations within only a few iterations.
It thereby avoids the computational cost associated with a large number of additional
iterations, whose primary purpose is to enforce binarity in the final reconstruction.

\item \textbf{Recovery of variance values:}
As shown in Figure~\ref{fig:SigmaEvolution}, the reconstruction of the true
variance value becomes less accurate for large values of $\hat\sigma$, close
to \numparam{$0.028$} (see the BLUE and RED curves).
For small values of $\hat\sigma$, close to \numparam{$0.024$}, the
variance reconstruction is also inaccurate (see the GREEN and BROWN curves).
The most accurate variance recovery is achieved for intermediate values of
$\hat\sigma$, around \numparam{$0.026$} (see the ORANGE and PURPLE curves).

It is worth noting that accurate reconstructions of the ground-truth barcode
can be obtained even in cases where the true standard deviation value $\hat\sigma$ is
not accurately recovered. This behavior is illustrated by the reconstructions
shown in Figures~\ref{fig:error_reconstructions_noise005}
and~\ref{fig:error_reconstructions_noise05} for the cases
\numparam{$\hat\sigma = 0.024$} and \numparam{$\hat\sigma = 0.028$}.


\item \textbf{Limitations of the proposed method}:
Despite its effectiveness under moderate noise and variance levels, Algorithm~1
exhibits certain limitations. In particular, its performance deteriorates
significantly when the variance parameter $\sigma^2$ exceeds specific thresholds,
even in low-noise scenarios. Under such conditions, the algorithm fails to produce
satisfactory reconstructions, leading to artifacts such as missing bars or the
appearance of spurious thin bars (see for example Figures~\ref{fig:experiment_sigma028noise10}
and~\ref{fig:error_reconstructions_noise10}). Based on our experiments, we find that the
reliability limit of the PCLS approach for the semi-blind deconvolution of the
barcode~$\hat{u}$ is approximately \numparam{$\sigma = 0.028$} for the standard deviation
and \numparam{$\delta = 5\%$} for the noise level. These findings are consistent
with the results reported in~\cite{E04}.
\end{itemize}

\section{Conclusions and future research}

We adopt the (PCLS) framework, considered by the authors in \cite{CLT13}, as a tool
for addressing the semi-blind deconvolution problem arising in the decoding of blurred
linear barcodes.
By parameterizing the barcode $u$ via a PCLS function $\phi$ and introducing
a slack variable associated with the partially unknown convolution kernel $K$,
we develop a solution strategy based on augmented Lagrangian techniques.

We establish several key theoretical results, most notably the existence of
generalized multipliers for the augmented Lagrangian functional under consideration
as well as the absence of duality gaps.
These findings enable us to show that the primal solutions generated by the
proposed augmented Lagrangian scheme provide stable approximate reconstructions
of the underlying barcode in the associated inverse problem, even in the presence
of noisy data.

We further conduct numerical experiments using an ADMM-type iterative algorithm,
considering a range of variance parameters (governing the degree of blurring)
and noise levels. The obtained results confirm the practical effectiveness of
the proposed method, while also illustrating its limitations.

Potential directions for future research include extending the PCLS framework
to the identification of 2D barcodes (e.g., QR code and Data Matrix),
as well as to semi-blind deconvolution problems involving non-Gaussian kernels $K$.

\section*{Acknowledgements}
ADC acknowledges support from FAPERGS (grant 123/2551-0001824-1) and CNPq (grant 401181/ 2025-1).
\\
EH  acknowledges support from DFG under Germany’s Excellence Strategy - EXC 2121
"Quantum Universe" (grant 390833306) and CNPq (grant 163469/2021-0).
\\
AL acknowledges support from CNPq (grant 307021/2025-4), FAPESC (grant 2024TR002238)
and the AvH Foundation.

\appendix 
\section*{Appendix A: Auxiliary results}

In this appendix, we recall some properties $\W$ 
in appropriated topologies that are used in the manuscript. Analogous results can be found in \cite{Agnelli2018, CLT13}. We present them here for sake of completeness. We also prove the point-wise continuity of the forward operator $B$.

First we briefly recall some basic facts about the space $\bv(\Omega)$. For a proof we refer the reader to \cite[Chapter~5]{EG92}.

\begin{lemma} \label{lemma:BV}
The following assertions hold true:
\medskip

\noindent {i)} The semi-norm $|\cdot|_\bv$ is weakly lower semi-continuous
with respect to $L^p$-convergence, i.e. if $\{x_k\} \in \bv(\Omega)$
converges to $x$ in the $L^p$-norm, then $x \in \bv(\Omega)$ and
$|x|_\bv \leq \liminf_{k\to\infty} |x_k|_\bv$.
\medskip

\noindent {ii)} $\bv(\Omega)$ is compactly embedded in $L^p(\Omega)$ for
$1 \leq p < d/(d-1)$. Consequently, any bounded sequence $\{x_k\} \in
\bv(\Omega)$ has a subsequence converging in $L^p(\Omega)$ to some
$x \in \bv(\Omega)$. \EH{d?}
\end{lemma}

The next lemma is devoted to the investigation of the relevant properties of operators $\W$. 
Such properties will be used throughout this paper. 

\begin{lemma} \label{lemma:calK-Ppc}
Let the doble-well operators $\W$ defined in Section~\ref{sec:1} and $P_{pc}$ the operator defined in \eqref{eq:def-pcls}. 
For $1\leq p < 2$, the following assertions hold true:
\begin{itemize}
\item[i)] $\W$ is continuous maps from $L^{4n}(\Omega)$ to
$L^2(\Omega)$ \\

\item[ii)] If $\| \W(\phi) \|_{L^1} = 0$ or $\| \W(\phi)
\|_{L^2} = 0$ for some $\phi \in L^{4n}(\Omega)$, then $\phi(x) \in \{
0, 1 \}$ a.e. in $\Omega$.





\end{itemize}
\end{lemma}

\begin{proof}
The continuity of $\W$ follows from the fact that  
\begin{align*}
\int_\Omega (\W(\phi) -\W(\psi) )^2 dx \ = &
 \ \int_\Omega (\phi^{n} - \psi^n)^2(\phi-1)^{2n} + 2(\phi^n - \psi^n)((\phi - 1)^n - (\psi-1)^n)(\phi- 1)^n \psi^n dx \\
 & +  \int_\Omega \psi^{2n}((\phi - 1)^n - (\psi-1)^n)^2dx\, ,  
\end{align*}
together with the fact that 
$(a^m - b^m) = (a-b)(a^{m-1} + a^{m-2}b+ \cdots a b^{m-2} + b^{m-1})$, 
the Cauchy-Schwarz inequality and continuous embedding of $L^{4n}(\Omega)$ in $L^2(\Omega)$. 

Assertions~ii) 
follows immediately form the definition of $\W$. 

\end{proof}

In the following lemma, we prove the weak continuity of the operator $P_{pc}$ and the
weak continuity on zero of $\W$.

\begin{lemma}\label{lemma:aux-1}
Let the operator $P_{pc}$ defined as in \eqref{eq:def-pcls}. Then:
\begin{itemize}
\item[ i)] For any sequence $\{\phi_k\}$ of $L^{4n}(\Omega)$ with $\phi_k
\rightharpoonup \phi$ in $L^{4n}(\Omega)$  (or $\phi_k \longrightarrow \phi$ in $L^{4n}(\Omega)$),
we have that $P_{pc}(\phi_k)
\rightharpoonup P_{pc}(\phi)$ in (or $P_{pc}(\phi_k) \longrightarrow P_{pc}(\phi)$) $L^2(\Omega)$.
In particular, this result follows for the bar code decoding application, where $P_{pc}$ is the identity.

\item[ ii)] For any sequence $\{\phi_k\}$ of $L^{4n}(\Omega)$ with $\phi_k \rightharpoonup \phi$ in $L^{4n}(\Omega)$ such that 
$\W(\phi_k) = 0$, we have that $\W(\phi)=0$.
\end{itemize}
\end{lemma}
\begin{proof}
Since $\Omega$ is bounded, we have that
$L^{4n}(\Omega) \subset L^2(\Omega) \subset L^1(\Omega)$. 
Therefore, $\phi_k \rightharpoonup \phi$ (or $\phi_k \longrightarrow \phi$) in $L^2(\Omega)$. 
Furthermore, $\phi_k \rightharpoonup \phi$ (or $\phi_k \longrightarrow \phi$) in $L^1(\Omega)$.

 By definition, the map $P_{pc}$ is affine in
$L^2(\Omega)$. Therefore, $P_{pc}(\phi_k) \rightharpoonup
P_{pc}(\phi)$ (or $P_{pc}(\phi_k) \longleftarrow
P_{pc}(\phi)$) in $L^2(\Omega)$.

To prove Item~ii), we remark that if $\W(\tilde{\phi}) = 0$ then
$\tilde{\phi}(\tilde{\phi} - 1) = 0$ a.e.. The reciprocity is also
true for any $\tilde{\phi} \in L^{4n}(\Omega)$. Now, the assertion follows from the weak
convergence of $\phi_k$ in $L^2(\Omega)$ and Lemma~\ref{lemma:calK-Ppc}~ii).
\end{proof}

Next, we demonstrate the continuity of the forward operator $B$.

\begin{prop}\label{th:B-continuos}[Continuity of $B(\cdot, \cdot)$]
Let the operator $B$ be defined defined as in~\eqref{eq:ip}. Furthermore, let $(\K_k, K_k, \phi_k)$
and $(\K, K, \phi)$ be any vectors of admissible functions, i.e., satisfying
Assumptions~\textbf{A1)}-\textbf{A2)}. If $K_k \to K$ in $L^2(\Omega)$ and
$\phi_k \to \phi$ in $L^{4n}(\Omega)$, then $B(K_k, P_{pc}(\phi_k)) \to
B(K, P_{pc}(\phi))$ in $L^2(\Omega)$.
\end{prop}
\begin{proof}
It follows from the assumptions and Lemma~\ref{lemma:aux-1}-i) that $P_{pc}(\phi_k)$ tends
to $P_{pc}(\phi)$ in $L^2(\Omega)$. Since $\Omega$ is bounded, this convergence is also
valid for $L^1(\Omega)$.
Consequently, it follows from Young's inequality for convolution~\cite{Ada75},
Assumption~\textbf{A1)}, and the convergence of $K_k$ to $K$ in $L^2(\Omega)$ that
\begin{align*}
\norm{B(K_k,P_{pc}(\phi_k))  - B(K, P_{pc}(\phi))}_{L^2(\Omega)} 
   & = \norm{K_k \ast P_{pc}(\phi_k) - K \ast P_{pc}(\phi)}_{L^2(\Omega)} \\
  & \hskip -2cm\leq \norm{K_k \ast (P_{pc}(\phi_k) - P_{pc}(\phi))}_{L^2(\Omega)} + \norm{(K_k - K)  \ast P_{pc}(\phi)}_{L^2(\Omega)} \\
  &  \hskip -2cm \leq \norm{K_k}_{L^2(\Omega)}\norm{P_{pc}(\phi_k) - P_{pc}(\phi)}_{L^1(\Omega)} + \norm{K_k - K}_{L^2(\Omega)}\norm{P_{pc}(\phi)}_{L^1(\Omega)} \,,
\end{align*}
concluding the assertion. 
\end{proof}

\bibliographystyle{amsplain}
\bibliography{lset-pcls}

\end{document}